\documentclass{amsart}

\usepackage{amsmath}
\usepackage{amsthm}
\usepackage{hyperref}
\usepackage{amsfonts,graphics,amsthm,amsfonts,amscd,latexsym}
\usepackage{epsfig}
\usepackage{flafter}
\usepackage{mathtools}
\usepackage{comment}
\usepackage{stmaryrd}
\usepackage[table,xcdraw]{xcolor}
\usepackage{url}
\usepackage[capitalise]{cleveref}

\hypersetup{
    colorlinks=true,    
    linkcolor=blue,          
    citecolor=blue,      
    filecolor=blue,      
    urlcolor=blue           
}
\usepackage{tikz}
\usepackage{tikz-cd}
\usetikzlibrary{graphs,positioning,arrows,shapes.misc,decorations.pathmorphing}

\tikzset{
    >=stealth,
    every picture/.style={thick},
    graphs/every graph/.style={empty nodes},
}

\tikzstyle{vertex}=[
    draw,
    circle,
    fill=black,
    inner sep=1pt,
    minimum width=5pt,
]
\usepackage[position=top]{subfig}
\usepackage{amssymb}
\usepackage{color}

\calclayout

\usetikzlibrary{decorations.pathmorphing}
\tikzstyle{printersafe}=[decoration={snake,amplitude=0pt}]

\newtheorem{theorem}{Theorem}[section]
\newtheorem{theorem-definition}{Theorem-Definition}[section]
\newtheorem{proposition}[theorem]{Proposition}
\newtheorem{corollary}[theorem]{Corollary}
\newtheorem{lemma}[theorem]{Lemma}

\newtheorem{conjecture}[theorem]{Conjecture}

\theoremstyle{definition}
\newtheorem{definition}[theorem]{Definition}

\newtheorem{remark}[theorem]{Remark}

\newtheorem{notation}[theorem]{Notation}

\newcommand{\on}[1]{\operatorname{#1}}

\newcommand{\Aut}{{\on{Aut}}}
\newcommand{\Gal}{{\on{Gal}}}
\newcommand{\Alb}{{\on{Alb}}}
\newcommand{\id}{{\on{id}}}
\newcommand{\Spec}{{\on{Spec}}}
\newcommand{\Proj}{{\on{Proj }}}
\newcommand{\GL}{{\on{GL}}}

\newcommand{\Supp}{\on{Supp}}

\newcommand{\codim}{{\on{codim }}}
\newcommand{\Pic}{{\on{Pic}}}

\newcommand{\NE}{\overline{NE}}

\newcommand{\hor}{\on{hor}}
\newcommand{\dist}{\on{dist}}
\newcommand{\interior}{\on{int}}

\newcommand{\WDiv}{\operatorname{WDiv}}

\newcommand{\coker}{\operatorname{coker}}

\newcommand{\Bir}{\operatorname{Bir}}
\newcommand{\PsAut}{\operatorname{PsAut}}
\newcommand{\Nef}{\overline{\mathcal{A}}}
\newcommand{\Mov}{\overline{\mathcal{M}}}

\newcommand{\Exc}{\operatorname{Exc}}

\newcommand{\pp}{\mathbb{P}}

\newcommand{\qq}{\mathbb{Q}}
\newcommand{\zz}{\mathbb{Z}}

\newcommand{\rr}{\mathbb{R}}

\newcommand{\xdasharrow}[2][->]{
\tikz[baseline=-\the\dimexpr\fontdimen22\textfont2\relax]{
\node[anchor=south,font=\scriptsize, inner ysep=1.5pt,outer xsep=2.2pt](x){#2};
\draw[shorten <=3.4pt,shorten >=3.4pt,dashed,#1](x.south west)--(x.south east);
}
}

\def\O#1.{\mathcal {O}_{#1}}			
\def\pr #1.{\mathbb P^{#1}}				
\def\af #1.{\mathbb A^{#1}}			
\def\ses#1.#2.#3.{0\to #1\to #2\to #3 \to 0}	
\def\xrar#1.{\xrightarrow{#1}}			
\def\K#1.{K_{#1}}						
\def\bA#1.{\mathbf{A}_{#1}}			
\def\bM#1.{\mathbf{M}_{#1}}				
\def\bL#1.{\mathbf{L}_{#1}}				
\def\bB#1.{\mathbf{B}_{#1}}				
\def\bK#1.{\mathbf{K}_{#1}}			
\def\subs#1.{_{#1}}					
\def\sups#1.{^{#1}}

\usepackage{tikz}
\usetikzlibrary{matrix,arrows,decorations.pathmorphing}

\numberwithin{equation}{section}

\usepackage[all]{xy}

\begin{document}

\title[The geometric Cone conjecture in relative dimension two]{The geometric cone Conjecture in relative dimension two}

\author[J.~Moraga]{Joaqu\'in Moraga}
\address{Department of Mathematics, UCLA, Los Angeles, CA 90095-1555, USA}
\email{jmoraga@math.ucla.edu}

\author[T.~Stark]{Talon Stark}
\address{Department of Mathematics, UCLA, Los Angeles, CA 90095-1555, USA}
\email{talonstark@math.ucla.edu}


\begin{abstract}
Let $X\rightarrow S$ be a fibration
of relative dimension at most two and
let $(X,\Delta)$ be a klt pair for which
$K_X+\Delta \equiv_S 0$.
We show that there are only finitely many
Mori chambers and Mori faces
in the movable effective cone $\mathcal{M}^e(X/S)$
up to the action of relative pseudo-automorphisms of $X/S$ preserving $\Delta$.
\end{abstract}

\maketitle

\setcounter{tocdepth}{1} 
\tableofcontents

\section{Introduction}

A fundamental perspective in algebraic geometry is that to understand an algebraic variety $X$, we must understand the morphisms from $X$. For projective varieties, by Stein factorization, every morphism factors as a morphism with connected fibers, a \emph{contraction}, followed by a finite morphism. Therefore, it is important to understand the contractions that a variety possesses. Contractions of a variety $X$ are classified by the faces of the nef effective cone $\mathcal{A}^e(X)$.

Smooth projective varieties come naturally in three flavors: Fano, Calabi--Yau, and canonically polarized. The Minimal Model Program (MMP) predicts that all varieties are built out of varieties of these three types. Fano varieties are Mori dream spaces, and in particular, their nef cones are rational polyhedral. For canonically polarized varieties, the behavior of the nef cone is completely intractable. On the other hand, the situation for Calabi--Yau varieties is intricate and has been the topic of a large body of research. Examples show that the nef cone of a Calabi--Yau variety ranges from being polyhedral to completely round.

The development of the minimal model program has shown that it is also important to understand the minimal models that a variety $X$ may possess. For smooth Fano and Calabi--Yau varieties, the minimal models are classified by the chambers of a wall--and--chamber decomposition of the movable effective cone $\mathcal{M}^e(X)$. The chamber of a minimal model is its nef effective cone under a canonical identification. Therefore, the wall--and--chamber decomposition of the movable effective cone provides a landscape for all minimal models and contractions thereof. 

In~\cite{Mor93,Mor96}, Morrison proposed a conjecture regarding the cone of nef divisors of a Calabi--Yau variety. He conjectured that the cone of nef effective divisors $\mathcal{A}^e(X)$ of a Calabi--Yau variety $X$ admits a rational polyhedral fundamental domain (RPFD) for the action of its automorphism group $\Aut(X)$. That is, a rational polyhedral cone $\Pi\subseteq \mathcal{A}^e(X)$ such that $\bigcup_{g\in \Aut(X)} g_*\Pi = \mathcal{A}^e(X)$ and $\interior{\Pi} \cap \interior{g_* \Pi} = \emptyset$ unless $g_* = \id$. In a similar vein, he conjectured that the cone of movable effective divisors $\mathcal{M}^e(X)$ of a Calabi--Yau variety $X$ admits a rational polyhedral fundamental domain for the action of the birational automorphism group $\Bir(X)$. The latter conjecture implies the first, as $\Aut(X)$ is the stabilizer of $\mathcal{A}^e(X)$ for the action $\Bir(X)\curvearrowright\mathcal{M}^e(X)$. 

The surprising aspect of the conjecture is that it predicts that a Calabi--Yau variety with non-polyhedral nef cone must have infinitely many automorphisms. In dimension two, the movable and nef cones coincide and minimal models are unique. In dimension two, Morrison's conjecture case was worked out by Sterk, Looijenga, and Namikawa with automorphisms constructed via the Torelli theorem \cite{Ste85, Nam85, Kaw97}. In dimension at least $3$, it remains unclear how to construct such automorphisms. Even though the conjecture has been settled for abelian varieties \cite{Pre12} and hyperk\"ahler  varieties \cite{AV17, Mar11}, and many other special cases, the conjecture is still wide open for Calabi--Yau manifolds in dimension at least $3$. In~\cite{Kaw97}, Kawamata proposed that both conjectures should be valid for Calabi--Yau fiber spaces $X\rightarrow S$, i.e., fibrations for which $K_X\equiv_S 0$. He showed that the geometric properties of the conjecture hold for Calabi--Yau fiber spaces $X\to S$ for $X$ a threefold and $S$ a curve or a surface. That is, there are finitely many contractions $X\to T \to S$ of $X$ over $S$ up to automorphisms commuting with the fibration $\Aut(X/S)$, and finitely many minimal models $X\dashrightarrow X' \to S$ over $S$ up to $\Bir(X/S)$.

In~\cite{Tot10}, Totaro showed that the cone conjecture is reasonable in the context of klt Calabi--Yau pairs, i.e., $(X,\Delta)$ is klt and $K_X+\Delta \equiv 0$. Further, he proposes that we have enough automorphisms when we restrict ourselves to those automorphisms $\Aut(X,\Delta)$ preserving the boundary $\Delta$. He gave the most modern formulation of the conjecture, proving it in dimension two. In this setting, $X$ may be non-terminal and so birational automorphisms may be not defined in codimension $1$. Hence, for a birational automorphism action on $N^1(X)$ to be defined, we must pass to the subgroup of birational automorphisms that are isomorphisms in codimension $1$, the \emph{pseudo-automorphisms} $\PsAut(X,\Delta)$. In addition, for $X$ terminal, minimal models are connected by flops \cite{Kaw08}, but for $X$ klt, minimal models may also be connected by crepant blowing ups. In this setting, the chambers of the movable cone classify all varieties we may obtain by sequences of flops, or \emph{marked small $\mathbb{Q}$-factorial modifications (SQMs)}. We refer to this conjecture as the Kawamata--Morrison--Totaro cone conjecture, or just the cone conjecture.

\begin{conjecture}[The Kawamata--Morrison--Totaro cone conjecture]
\label{conj:KMT-cone}
Let $f:\:(X,\Delta)\to S$ be a klt Calabi--Yau pair. That is, $(X,\Delta)$ is klt and $K_X+\Delta \equiv_S 0$. Then, the following statements hold:
\begin{enumerate}
    \item 
    
    The group $\Aut(X/S,\Delta)$ acts on the nef effective cone $\mathcal{A}^e(X/S)$ with finitely many orbits of Mori faces. The group $\PsAut(X/S,\Delta)$ acts on the movable effective cone $\mathcal{M}^e(X/S)$ with finitely many orbits of Mori chambers and faces. Equivalently, there are finitely many contractions of $X$ over $S$ up to $\Aut(X/S,\Delta)$ and finitely many marked SQMs of $X$ over $S$ up to $\PsAut(X/S,\Delta)$. 
    \item 
    
    The above finiteness statements are due to the existence of rational polyhedral fundamental domains $\Pi\subseteq \mathcal{A}^e(X/S)$ and $\Pi'\subseteq \mathcal{M}^e(X/S)$ for the actions of $\Aut(X/S,\Delta)\curvearrowright \mathcal{A}^e(X)$ and $\PsAut(X/S,\Delta)\curvearrowright \mathcal{M}^e(X)$, respectively.
\end{enumerate}
\end{conjecture}

We refer to (1) as the geometric cone conjecture and (2) as the existence of a RPFD.

In~\cite{FHS21}, the authors proved the geometric cone conjecture for all terminal Calabi--Yau fiber spaces of relative dimension $1$, allowing for the total space to have arbitrary dimension. They used this as the primary tool to prove the landmark result that elliptic Calabi--Yau $3$-folds form a bounded family, implying that elliptic Calabi--Yau $3$-folds have only finitely many topological types. The geometric cone conjecture will be also one of the primary tools to prove the boundedness of Calabi--Yau manifolds admitting special fibrations in higher dimensions.

In this article, we prove the geometric cone conjecture for all klt Calabi--Yau pairs of relative dimension at most two, simultaneously extending the results of~\cite{Kaw97},~\cite{Tot10},~\cite{FHS21},~\cite{LZ22}, and~\cite{Li23}. 

\begin{theorem}
\label{thm:mainresult}
The geometric cone conjecture holds in relative dimension at most two.
\end{theorem}

All minimal models of a klt pair with $K_X+\Delta$ pseudo-effective are klt Calabi--Yau over the same ample model $\Proj(R(X,K_X + \Delta))$, and this result gives the immediate application:

\begin{corollary}
\label{cor:minmodels}
    Let $(X,\Delta)$ be a klt pair with $\kappa(X,\Delta)\geq \dim(X)-2$, then $(X,\Delta)$ possesses only finitely many isomorphism classes of minimal models.
\end{corollary}

It is worth mentioning that Li and Zhao have obtained the previous results for terminal pairs in relative dimension two (see~\cite{LZ22,Li23}).
Our approach is different -- however, in both works we need to understand 
the cone conjecture for the generic fiber of the relative log Calabi--Yau pair.
To prove~\cref{thm:mainresult}, we must prove that there are finitely many Mori faces and chambers of $\mathcal{M}^e(X/S)$ up to the action of $\PsAut(X/S,\Delta)$. The strategy to do this consists of the following five steps:
\begin{enumerate}
    \item extend the result of \cite[Theorem 3.4]{FHS21} from terminal CY fibrations of relative dimension $1$ to klt CY pairs of relative dimension $1$;
    \item prove the conjecture for the generic fiber i.e., \cite[Theorem 4.1]{Tot10} but for non-closed fields of characteristic $0$;
    \item use the first two steps to show that, for $(X/S,\Delta)$ klt CY of relative dimension two, there are finitely many $\PsAut(X/S,\Delta)$-orbits of faces and chambers in $\mathcal{M}^e(X/S)$ corresponding to fiber space structures;
    \item use induction on dimension of the space of vertical divisors $V(X/S)$ and the relative dimension 1 statement to show that there are finitely many orbits of chambers and faces containing faces that correspond to fiber space structures; and
    \item show that the remainder of the orbits of faces and chambers is finite up to $\PsAut(X/S,\Delta)$ using structural results about the generic fiber and making a careful topological analysis of certain subsets of $\mathcal{M}^e(X/S)$.
\end{enumerate}

Step 1 is achieved in Section~\ref{subsec:rel-dim-1}.
In Section~\ref{section4}, we work over non-closed fields of characteristic zero to settle Step 2.
In Section~\ref{sec:finiteness-fiber-space}, we show the finiteness of faces and chambers of $\mathcal{M}^e(X/S)$ corresponding to fiber space structures, finishing Step 3.
Step 4 is proved in the first few lines of 
the proof of Theorem~\ref{thm:mainresult2}.
Finally, Step 5 is achieved in Theorem~\ref{thm:mainresult2}.

\section{Preliminaries}

In this section, we state some preliminary definitions and results. Unless otherwise stated, we work with normal varieties over an algebraically closed field $k$ of characteristic zero. 

\begin{definition}
    A projective morphism of varieties $f:\:X \to S$ is called a \emph{contraction} if $f_*\mathcal{O}_X = \mathcal{O}_S$ (in particular, it follows that $f$ is surjective). For normal projective varieties over $\mathbb{C}$, this is equivalent to being a surjective morphism with connected fibers.
\end{definition}

There is also a notion of rational map being a contraction (see \cite{HK00}):

\begin{definition}
 Let $f:\: X \dashrightarrow Y$ be a rational map between normal projective varieties. Let $(p, q):\: W \to X \times Y$ be a resolution of $f$ with $W$ projective, and $p$ birational. We say that the map $f$ \emph{has connected fibres} if $q$ does. If $f$ is birational, we call it a \emph{birational contraction} if every $p$-exceptional divisor is $q$-exceptional. For a $\mathbb{Q}$-Cartier divisor in $Y$, $f^*(D)$ is defined to be $p^*(q^*(D))$, and all of these are independent of the resolution.
 An effective divisor $E$ on $W$ is called \emph{$q$-fixed} if no effective Cartier divisor whose support is contained in the support of $E$ is $q$-moving, i.e., for every such divisor $D$ the natural map 
 $$\mathcal{O}_Y \to q_*(\mathcal{O}(D))$$
 is an isomorphism. A map $f$ is called a \emph{contraction} if every $p$-exceptional divisor is $q$-fixed. An effective divisor $E \subset X$ is called $f$-fixed if any effective divisor of $W$ supported on the union of the strict transform of $E$ with the exceptional divisor of $p$ is $q$-fixed.
\end{definition}

\subsection{Relative properties of divisors}
In this subsection, we recall some relative properties of divisors and the concept of relative log pairs. 

For a Cartier divisor $D$ on a variety $X$, the section ring $R(X,D)$ is the graded ring 
$$R(X,D) = \bigoplus_{m\in \mathbb{N}} H^0(X,mD),$$
where the multiplication happens in the function field of $X$.
If $R(X,D)$ is finitely generated and $D$ is effective, then there is an induced rational contraction $f_D:\: X \dashrightarrow \Proj(R(X,D))$ that is regular outside the stable base locus $\mathbb{B}(D):=\bigcap_{m>0} \text{Bs}(mD)$. We call this map the \emph{Iitaka fibration} of $D$~\cite[Theorem 2.1.33]{Laz04a}.

\begin{definition}
    For two effective $\mathbb{Q}$-Cartier divisors $D_1$ and $D_2$, we say that they are \emph{Mori equivalent} if the maps $f_{D_1}$ and $f_{D_2}$ are the same. The \emph{N\'eron-Severi space} $N^1(X)$ denotes the $\mathbb{R}$-vector space of $\mathbb{R}$-Cartier divisor classes modulo numerical equivalence. By a \emph{Mori chamber} (resp. \emph{Mori face}), of $N^1(X)$, we mean the closure of a Mori equivalence class whose interior is non-empty (resp. empty). We frequently refer to these as just \emph{chambers} and \emph{faces}. In \cref{rem:faceschambers}, we will see that this is shall cause no ambiguity. 
\end{definition}

We now define some relative properties of divisors that will be of primary interest:
\begin{definition}
    For $f:\: X\to S$ a projective contraction of normal varieties, a Cartier divisor $D$ on $X$ is said to be:
    \begin{itemize}
        \item \emph{$f$-nef} if $(D\cdot C)\geq 0$ for all curves $C$ contained in a fiber of $f$,
        \item \emph{$f$-movable} if $\codim \Supp \coker(f^* f_* \mathcal{O}_X(D) \to \mathcal{O}_X(D))\geq 2$,
        \item \emph{$f$-effective} if $f_*\mathcal{O}_X(D)\neq 0$,
        \item \emph{$f$-big} if $\kappa(X_\eta, D_\eta) = \dim X - \dim S$ for the generic point $\eta\in S$,
        \item \emph{$f$-vertical} if $f(D)\neq S$,
        \item \emph{$f$-basepoint free} if $f^* f_* \mathcal{O}_X(D) \to \mathcal{O}_X(D)$ is surjective,
        \item \emph{$f$-semiample} if $mD$ is $f$-basepoint free for $m$ sufficiently divisible and large.
    \end{itemize}
\end{definition}

If the map $f:\: X\to S$ is implicitly understood, we may say ``nef over $S$", ``movable over $S$", etc. Taking $f:\: X\to \Spec(k)$, we recover the usual notions of nefness, movability, etc. 

We also have a relative version of the section ring
\[
\mathcal{R}(X/S,D):= \bigoplus_{m\in \mathbb{N}} f_*\mathcal{O}_X(mD).
\]
If $\mathcal{R}(X/S,D)$ is finitely generated as an $\mathcal{O}_S$-algebra and $D$ is an $f$-effective Cartier divisor, we get a relative Iitaka fibration $X\dashrightarrow {\rm Proj}_S\mathcal{R}(X/S,D)$ that is regular outside the relative stable base locus 
\[
\mathbb{B}_S(D):=\bigcap_{m>0}\Supp \coker(f^* f_* \mathcal{O}_X(mD) \to \mathcal{O}_X(mD)).
\]

Hence, if $D$ is $f$-semiample, we get an associated morphism $X\to T:={\rm Proj}_S\mathcal{R}(X/S,D)$ over $S$, and some multiple of $D$ is the pullback of an \emph{$f$-ample} divisor on $T$. A Cartier divisor $D$ is an $f$-ample divisor if its relative Iitaka fibration is an embedding. Conversely, for any contraction $\phi:\: X\to T$ over $S$, $\phi$ is the Iitaka fibration of the $f$-semiample divisor $\phi^*A$ for any choice of $f$-ample divisor $A$ on $T$. 

By the \emph{relative N\'eron-Severi space}, we mean the real vector space $N^1(X/S)$ spanned by Cartier divisors on $X$ modulo numerical equivalence on curves contained in the fibers of $f$ (i.e., $D_1\equiv_S D_2$ means that $D_1\cdot C = D_2 \cdot C$ for any curve $C$ mapped to a point under $f$). We define relative notions of Mori equivalence, Mori chambers, and Mori faces inside of $N^1(X/S)$ just as in $N^1(X)$. Inside the relative N\'eron-Severi space $N^1(X/S)$, we have the following convex cones of divisors. 

The \emph{$f$-nef} cone $\Nef(X/S)$, (resp. \emph{$f$-movable cone} $\Mov(X/S)$, resp. \emph{$f$-pseudoeffective cone} $\overline{\mathcal{B}}(X/S)$) is the closed convex cone in $N^1(X/S)$ generated by the numerical classes of $f$-nef (resp. $f$-movable, resp. $f$-effective) divisors. These cones satisfy the containments 
\[
\Nef(X/S)\subseteq \Mov(X/S) \subseteq \overline{\mathcal{B}}(X/S).
\] 
The interior of the $f$-nef cone $\mathcal{A}(X/S)$ is the \emph{$f$-ample cone} generated by numerical classes of $f$-ample divisors and the interior of the $f$-pseudoeffective cone $\mathcal{B}(X/S)$ is the \emph{$f$-big} cone generated by numerical classes of $f$-big divisors. The $f$-effective cone $\mathcal{B}^e(X/S)$ generated by numerical classes of $f$-effective is not necessarily open nor closed, and it lies in between the $f$-big and $f$-pseudoeffective cones. We write 
\[
\mathcal{A}^e(X/S):= \Nef(X/S) \cap \mathcal{B}^e(X/S) \qquad \text{ and } 
\qquad 
\mathcal{M}^e(X/S):= \Mov(X/S) \cap \mathcal{B}^e(X/S),
\]
for the $f$-effective $f$-nef cone 
and the $f$-movable $f$-effective cone,
respectively. We denote by $V(X/S)$ the subspace of $N^1(X/S)$ generated by $f$-vertical divisors and the by $v(X/S)$ the dimension of $V(X/S)$. 

Now, we turn to recall the concept of relative pairs.

\begin{definition}
    A \emph{(relative) log pair}, which we denote $f:\:(X,\Delta)\to S$ or $(X/S,\Delta)$, consists of following data:
    \begin{itemize}
        \item a normal variety $X$, 
        \item a projective contraction $f:\: X\rightarrow S$, 
        \item and an effective $\mathbb{R}$-divisor $\Delta$ such that $K_X+\Delta$ is $\rr$-Cartier.
    \end{itemize} 
    
    The \emph{relative dimension} of $(X/S,\Delta)$ is simply $\dim X - \dim S$. We may refer to $X\rightarrow S$ as the \emph{structure morphism} of the relative pair. We say that $(X/S,\Delta)$ is a \emph{klt pair} (resp. terminal, canonical) if $(X,\Delta)$ is klt (resp. terminal, canonical).
\end{definition}

Denote by $N_1(X/S)$ the finite-dimensional $\mathbb{R}$-vector space generated by classes of curves mapping to a point under $f$, modulo numerical equivalence. Define $\NE(X/S)$, the \emph{Mori cone} of $X/S$, to be the convex cone generated by classes of curves. 
A \emph{convex cone} is a convex subset $C$ for which $x\in C$ implies that $\mathbb{R}^{\geq0}\cdot x\subset C$. There is a perfect pairing $N_1(X/S) \times N^1(X/S)\to \mathbb{R}$ induced by intersection. Under this pairing, it is clear from the definitions that the cones $\Nef(X/S)$ and $\NE(X/S)$ are dual to one another. 

\subsection{Birational geometry, surfaces, and klt Calabi--Yau pairs}

In this subsection, we discuss some ubiquitous results in birational geometry that we will need as well as some fundamental results about klt Calabi--Yau pairs.
We will frequently make use of the following lemma \cite[Lemma 2.31]{KM98}.

\begin{lemma}
    Let $h: X \to Y$ be a proper birational morphism between normal varieties. Let $-D$ be an $h$-nef $\mathbb{R}$-Cartier $\mathbb{R}$-divisor on $X$. Then we have the following statements:
    \begin{enumerate}
        \item $D$ is effective if and only if $h^*D$ is.
        \item Assume that $D$ is effective. Then for every $y \in Y$, either $h^{-1}(y) \subset \Supp D$ or $h^{-1}(y) \cap \Supp B = \emptyset$.
    \end{enumerate}
\end{lemma}

In the case of surfaces, this is essentially a basic result about quadratic forms (see \cite[V.2.3.5]{Bou71}):
\begin{lemma}
    Let $N$ be a set of curves on a smooth projective surface on which the intersection pairing is negative definite. Let $-D$ be a linear combination of the curves $N_i$ which has non-negative intersection with each $N_i$. Then $D$ is effective. Moreover, the support of $D$ is a union of some connected components of $N$.
\end{lemma}

The following is a result of Ambro and Kawamata \cite[Theorem 4.1]{Amb05}:

\begin{theorem}
    If $f:\: (X,\Delta)\to S$ is a klt CY pair and $X\to T$ a contraction over $S$, there exists an effective $\mathbb{R}$-divisor $\Delta_T$ on $T$ such that $(T/S,\Delta_T)$ is klt CY. 
\end{theorem}

\begin{remark}
    The target of a contraction $X\to T$ from a klt CY pair $(X/S,\Delta)$ is frequently not $\mathbb{Q}$-factorial. When we invoke the canonical bundle formula to find an effective $\mathbb{R}$-divisor $\Delta_T$ such that $(T/S,\Delta_T)$ is ``klt CY", it will be understood that we are dropping the $\mathbb{Q}$-factorial assumption in this case. 
\end{remark}

The following is a consequence of the cone theorem together with the fact that we may decompose $f$-big divisors into the sum of an $f$-effective and an $f$-ample divisor (a detailed proof can be found in \cite[Theorem 5.7]{Kaw88}):

\begin{proposition}
\label{prop:locallypolyhedral}
    Let $(X/S,\Delta)$ be a klt CY pair over $S$. The cone $$\overline{\mathcal{A}}(X/S)\cap \mathcal{B}(X/S) = \mathcal{A}^e(X/S)\cap \mathcal{B}(X/S)$$ is locally rational polyhedral inside of the open cone $\mathcal{B}(X/S)$. This means that the intersection of $\mathcal{A}^e(X/S)\cap \mathcal{B}(X/S)$ with any any rational polyhedral cone $\Sigma\subset \mathcal{B}(X/S)\cup \{0\}$ is rational polyhedral. Any face $F$ of this cone corresponds to a birational contraction $\phi:\: X\to Y$ over $S$ by the equality $F=\phi^*(\overline{\mathcal{A}}(Y/S)\cap \mathcal{B}(Y/S))$.
\end{proposition}

The following consequence of \cite[Theorem E]{BCHM10} is valid with no assumption on $\mathbb{Q}$-factoriality and proves the geometric cone conjecture in relative dimension $0$, that is, when $f:(X,\Delta)\to S$ is birational. 
\begin{lemma}
\label{lem:crepant}
    Let $(S,\Delta_S)$ be a klt pair. Then, the set of all projective birational morphisms of normal varieties $\nu:\: T\to S$ such that if $K_{T}+\Delta_{T} = \nu^*(K_S + \Delta_S)$, then $\Delta_{T}\geq 0$, is finite.
    
\end{lemma}

\begin{proof}
    If $(S',\Delta')$ is a $\mathbb{Q}$-factorial terminalization of $(S,\Delta_S)$, then any such model $\nu:\: T\to S$ is obtained as a rational contraction of $(S',\Delta')$. Moreover, $(T,\Delta_{T})$ is a weak log canonical model of $(S',\Delta')$, so the theorem follows from \cite[Theorem E]{BCHM10}.
\end{proof}

In what follows, we will need a relative version of the abundance conjecture.
The abundance conjecture is known to hold in relative dimension at most three~\cite[Corollary 4.7.9]{Fuj17}:

\begin{theorem}
\label{thm:abundance}
    Let $f:\: X \to S$ be a projective morphism of relative dimension at most three and $\Delta$ an effective $\mathbb{R}$-divisor such that $(X,\Delta)$ is klt. If $K_X+\Delta$ is $f$-nef, then $K_X + \Delta$ is $f$-semiample.
\end{theorem}

The abundance conjecture is also known in any dimension for numerically trivial log canonical divisors (see \cite[Theorem 4.2]{Amb05}, \cite[Corollary 4.9]{Nak04}, or \cite[Theorem 1.2]{Gon13} for the result on lc pairs):

\begin{theorem}
\label{thm:index}
    Let $(X/S,\Delta)$ be a klt CY pair with rational coefficients, then $K_X + \Delta\sim_{\mathbb{Q},S} 0$. Hence, there exists an positive integer $I$, the \emph{index} of $(X,\Delta)$, such that $I(K_X+\Delta)\sim_S 0$. 
\end{theorem}

Frequently, we will pass to a cover that improves the singularities of our variety.

\begin{definition}
    Let be $(X,\Delta)$ A klt Calabi--Yau pair with standard coefficients (i.e., the coefficients of $\Delta$ are of the form $1-\frac{1}{a_i}$ with $a_i\in \zz_{\geq 1}$) such that $K_X+\Delta\sim_\mathbb{Q} 0$. Then there exists a positive integer $I$, the \emph{index} of $(X,\Delta)$, such that $I(K_X+\Delta)\sim 0$, and the cyclic cover 
    \[
    Z:={\rm Spec}_X\big{(}\bigoplus_{i=0}^{I-1} \mathcal{O}_X(-i(K_X+\Delta))\big{)} \to X
    \]
    has the property that $Z$ has canonical singularities and $K_Z\sim 0$. The cover $Z\to X$ is called the \emph{index one cover} of $(X,\Delta)$.
\end{definition}

By \cref{thm:index}, we see that this $\mathbb{Q}$-triviality assumption $K_X+\Delta\sim_\mathbb{Q} 0$ is always true for klt CY pairs $(X,\Delta)$ whose boundary  $\Delta$ is a $\mathbb{Q}$-divisor.

\subsection{Minimal models, Small Q-factorial modifications, and SQM contractions}
In this subsection, we recall some definitions and properties of minimal models and SQM contractions.

\begin{definition}
    Let $f:\:(X,\Delta)\to S$ be a klt pair with $K_X+\Delta$ $f$-pseudoeffective. A \emph{marked minimal model} of $(X/S,\Delta)$ is a relative log pair $(X_0/S, \Delta_0)$ and a map $\alpha:\: X \dashrightarrow X_0$ satisfying the following conditions:
    \begin{enumerate}
        \item the variety $X_0$ is $\qq$-factorial,
        \item $\alpha$ is a birational map $\alpha:\: X\dashrightarrow X_0$ over $S$ that does not extract any divisors,
        \item $\Delta_0 = \alpha_{0*}\Delta$,
        \item $K_{X_0} + \Delta_0$ is $f$-nef, and
        \item $a(E,X,\Delta)\leq a(E, X_0, \Delta_0)$ for every $\alpha$-exceptional divisor $E\subset X$.
    \end{enumerate} 
\end{definition}

\begin{definition}
    Let $X\to S$ be a projective contraction of normal varieties. A \emph{marked small $\qq$-factorial modification} of $X/S$ is a couple $(X_0, \alpha)$ satisfying the following conditions:
    \begin{enumerate}
        \item the variety $X_0$ is $\qq$-factorial, and
        \item $\alpha$ is a birational map $\alpha:\: X\dashrightarrow X_0$ over $S$ that does not contract or extract any divisors.
    \end{enumerate} 
    An \emph{SQM contraction} of $X/S$ is a projective contraction $X_0 \to T$ from a marked SQM $\alpha_0:\: X\dashrightarrow X_0$ of $X/S$. We will also loosely call the composition $X\dashrightarrow X_0 \to T$ an SQM contraction of $X/S$ when the domain is unambiguous. If $\dim X = \dim T$, we will call this a \emph{birational SQM contraction} and if $\dim X > \dim T$ we will call this a \emph{SQM fiber space structure}. 
\end{definition}

For $(X/S,\Delta)$ a terminal CY pair the notions of marked minimal models and marked SQMs coincide. Indeed, by~\cite{Kaw08}, all minimal models are connected by flops. On the other hand, if $(X,\Delta)$ is a klt CY pair, minimal models may also be connected by crepant blowups, so marked SQMs are marked minimal models, but not necessarily vice-versa.

The following lemma shows that any birational contraction between klt CY pairs is an SQM contraction. It is valid even if the pairs are not $\mathbb{Q}$-factorial.

\begin{lemma}
\label{lem:birationalcontraction}
Let $(X/S,\Delta_X)$ and $(Y/S,\Delta_Y)$ be klt CY pairs. Let $\pi\colon X\dashrightarrow Y$ be a birational contraction over $S$ for which $\pi_*\Delta_X=\Delta_Y$. The map $\pi$ is a birational SQM contraction of $X/S$.
\end{lemma}

\begin{proof}
   Since $\pi$ doesn't extract any divisors, any divisorial valuation on $(Y/S,\Delta)$ is a divisorial valuation on $(X/S,\Delta_X)$. We may extract the divisorial valuations over $(Y/S,\Delta_Y)$ that have divisorial center on $(X/S,\Delta_X)$ but not on $(Y/S,\Delta_Y)$ to obtain a model $(Y'/S,\Delta')$ that is $\mathbb{Q}$-factorial and a projective birational contraction $\pi':\: Y' \to Y$ over $S$ such that $\pi'_* \Delta' = \Delta_Y$ (see, e.g., ~\cite[Lemma 1.4.3]{BCHM10}).
    The induced birational map $\alpha:\: X\dashrightarrow Y'$ is a small $\mathbb{Q}$-factorial modification. Hence, $\pi$ is the composition of a small birational map and a projective birational contraction. 
\end{proof}

\begin{remark}
   While \cref{lem:birationalcontraction} shows that a birational contraction $(X/S,\Delta)\dashrightarrow (T/S,\Delta_T)$ is a composition of a marked SQM $\alpha: X\dashrightarrow X_0$ and a projective contraction $X_0\to T$, the same need not be true for rational contractions $(X/S,\Delta)\dashrightarrow (T/S,\Delta_T)$ of positive relative dimension, in general.
\end{remark}

We now list some known results on the structure of the previously mentioned cones. 

\begin{proposition}
    \label{prop:movdecomp}
    Assume the termination of flips in relative dimension $n$. Let $(X/S,\Delta)$ be a klt Calabi--Yau pair of relative dimension $n$ over $S$. Let $D$ be an $\rr$-divisor with $[D]\in \mathcal{M}^e(X/S)$. Then there exists a marked SQM $\alpha\colon X\dashrightarrow X'$ over $S$ such that $\alpha_*D$ is $f$-nef over $S$. This gives us a decomposition $$\mathcal{M}^e(X/S) = \bigcup_{(X_0/S,\alpha_0)} \alpha_0^*\mathcal{A}^e(X_0/S)$$ where the union is taken over all marked SQMs $\alpha_0:\: X\dashrightarrow X_0$ over $S$.
\end{proposition}

\begin{proof}
    This result follows from the MMP. If we take a divisor class $D\in \mathcal{M}^e(X/S)$, the pair $(X,\Delta + \epsilon D)$ is still klt because the klt condition is open, and we may run a $(K_X+\Delta + \epsilon D)$-MMP over $S$. The MMP will contract the relative diminished base locus of $K_X+\Delta + \epsilon D\equiv_S \epsilon D$. Since $D$ is $f$-movable, this MMP consists of a sequence of flops. We also have that $\alpha_1^*\mathcal{A}(X_1/S)\cap \alpha_2^*\mathcal{A}(X_2/S)\neq \emptyset$ implies that the marked SQMs $(X_1/S,\alpha_1)$ and $(X_2/S,\alpha_2)$ are equal. Indeed, if $D\in \alpha_1^*\mathcal{A}(X_1/S)\cap \alpha_2^*\mathcal{A}(X_2/S)$, then
    \[
    X_1 = {\rm Proj}_S(\mathcal{R}(X_1/S, D) \cong {\rm Proj}_S(\mathcal{R}(X_2/S, D)) = X_2
    \]
    with the isomorphism being compatible with $\alpha_1$ and $\alpha_2$.
\end{proof}

\begin{remark}
\label{rem:faceschambers}
    Assuming the termination of flips, the above result, together with the cone theorem, implies that the Mori chambers of $\mathcal{M}^e(X/S)$ are given by $\alpha_0^*\mathcal{A}^e(X_0/S)$ for some marked SQM $\alpha_0:\: X \dashrightarrow X_0$ and the Mori faces are given by $\alpha_0^*\phi^*\mathcal{A}^e(Y/S)$ for some SQM contraction $\phi:\: X_0 \to T$ on some marked SQM $\alpha_0:\: X\dashrightarrow X_0$. This result also implies that marked SQMs may be decomposed into sequences of flops. 
\end{remark}

\begin{proposition}
    \label{prop:Qgenerated}
    Assume the abundance conjecture in relative dimension $n$. Let $f:\: (X,\Delta)\to S$ be a klt Calabi--Yau pair of relative dimension $n$ over $S$. The cones $\mathcal{A}^e(X/S)$ and $\mathcal{M}^e(X/S)$ are generated by the numerical classes of $\mathbb{Q}$-Cartier divisors. 
\end{proposition}

\begin{proof}
    By the abundance conjecture, any $f$-effective $f$-nef divisor is the pullback of a $\mathbb{R}$-Cartier ample divisor. The ample cone is generated by classes of Cartier divisors, so the statement holds for $\mathcal{A}^e(X/S)$. Since every $f$-effective $f$-movable divisor becomes $f$-effective $f$-nef after a finite sequence of flops, the same holds for $\mathcal{M}^e(X/S)$. 
\end{proof}

The following is an extension of \cite[Theorem 2.6]{Kaw97} to klt CY pairs:
\begin{proposition}
\label{prop:locallyfinite}
    Let $f:\: (X,\Delta)\to S$ be a klt CY pair of relative dimension $n$. Then, the decomposition $$\mathcal{M}^e(X/S)\cap \mathcal{B}(X/S) = \bigcup_{(X_0,\alpha_0)} 
    \alpha_0^*\mathcal{A}^e(X_0/S) \cap \mathcal{B}(X/S)$$ is locally finite inside of $\mathcal{B}(X/S)$, meaning that for any closed convex cone $\Sigma\subset \mathcal{B}(X/S)\cup \{0\}$, there exist only a finite number of chamber and faces of $\mathcal{M}^e(X/S)\cap \mathcal{B}(X/S)$ that intersect $\Sigma$.
\end{proposition}

\begin{proof}
     Let $[D]\in \mathcal{M}^e(X/S)\cap \mathcal{B}(X/S)$. Since $D$ is $f$-big, on a general fiber $f^{-1}(y)$ the $(K_{X_{f^{-1}(y)}}+\Delta_{f^{-1}(y)} + \epsilon D_{f^{-1}(y)})$-MMP has a good minimal model by \cite{BCHM10} that extends to a good minimal model $\alpha_0:\: X\dashrightarrow X_0$ for the $(K_X+\Delta+\epsilon D)$-MMP over $S$ by \cite[Theorem 2.12]{HX15}. Let $F$ be the face of $\alpha_0^*\mathcal{A}^e(X'/S)$ that contains $[D]$ in its interior.  Let $\phi:\: X_0 \to Y$ be the contraction over $S$ associated to the face $F = \phi^*\mathcal{A}^e(Y/S)$. Since $[D]$ is $f$-big, $\dim X = \dim Y$. The relative pair $(Y/S,\phi_*\Delta)$ is a klt CY and by \cref{lem:crepant}, there exists only finitely many faces and chambers $F_i= \alpha_i^*\phi_i^*\mathcal{A}^e(X_i/S)$ containing the face $F$. Since each of these faces and chambers are locally rational polyhedral inside of $\mathcal{B}(X/S)$ by \cref{prop:locallypolyhedral}, there exists a small open cone $\Sigma$ containing $[D]$ such that $\Sigma\cap \mathcal{M}^e(X/S)= \Sigma\cap \bigcup_i F_i$. The result follows because any closed convex cone $\Sigma\subset \mathcal{B}(X/S)\cup \{0\}$ is compact modulo scaling. 
\end{proof}

We will also need the following technical lemma that is a slight generalization of \cite[2.26]{DCS21}. 

\begin{lemma}
\label{lem:existssqm} 
Assume the abundance conjecture in relative dimension $n$. Let $f:\:(X,\Delta)\to S$ be a klt Calabi--Yau pair of relative dimension $n$. Assume that $S$ is $\mathbb{Q}$-factorial and let $\pi_S:\: S\dashrightarrow S_0$ a birational contraction. Then, there exists a marked SQM $\alpha_0 \colon X\dashrightarrow X_0$ and a projective contraction $f_0\colon X_0\rightarrow S_0$ making the following diagram commutative:
    \begin{center}
    \begin{large}
    \begin{tikzcd}
        X \arrow[d, "f", swap] \arrow[r,"\alpha_0",dashed]
        & X_0 \arrow[d, "f_0"] \\ 
        S \arrow [r, "\pi_S", swap, dashed]
        & S_0
    \end{tikzcd}
    \end{large}
    \end{center}  
\end{lemma}

The proof in \cite[Proposition 2.26]{DCS21} works for klt CY pairs, if instead of passing to a terminalization at the end of the proof, one uses \cite[Prop 1.4.3]{BCHM10} to extract the correct non-terminal exceptional valuations and produce the marked SQM $\alpha_0:\:X\dashrightarrow X_0$.

We emphasize that the statemens:~\cref{prop:movdecomp},~\cref{prop:Qgenerated}, and~\cref{lem:existssqm} are valid up to relative dimension three due to the previous mentioned theorems on abundance (\cref{thm:abundance}) and termination of flips (\cite[Theorem 1]{Kaw92}, \cite[Theorem 2.12]{HX15}).

\subsection{Exceptional and degenerate divisors}
In this subsection, we recall some some special divisors for fibrations, the so-called exceptional divisors and degenerate divisors.
We also prove some statements about these divisors.

\begin{definition}
Let $f\colon X\rightarrow S$ be a proper surjective morphism of normal varieties. 
We say that a prime divisor $E$ on $X$ is \emph{$f$-exceptional} if there exists a marked SQM $\alpha:\: X\dashrightarrow X_0$ over $S$
and a divisorial contraction $X_0\rightarrow Y$ over $S$ whose exceptional divisor is the strict transform of $E$ on $X_0$.
\end{definition}

The following definition comes from~\cite[Definition 2.8]{Lai11}.

\begin{definition} 
    Let $f:\: X\to S$ be a proper surjective morphism of normal varieties and $D\in \WDiv_\mathbb{R}(X)$ an effective Weil divisor. $D$ is \emph{$f$-degenerate} if either:
    \begin{enumerate}
        \item we have $\codim(\Supp(f(D)))\geq 2$; or
        \item $D$ there exists a prime divisor $\Gamma\not\subseteq D$ such that $f(\Gamma)\subseteq f(D)$ has codimension $1$ in $Y$.
    \end{enumerate}
    If the second condition above holds, then we say that $D$ is \emph{of insufficient fiber type}.
\end{definition}

\begin{lemma}
\label{lem:degenerate}
    Let $f\colon (X,\Delta)\rightarrow S$ be a klt Calabi--Yau pair.
    Let $D$ be a prime divisor on $X$.
    The divisor $D$ is $f$-exceptional if and only if it is $f$-degenerate.
    Furthermore, there are finitely many such divisors.
\end{lemma}

\begin{proof}
    If $D$ is a prime degenerate divisor on $X$ over $S$, then $D$ is contained in its diminished base locus $\mathbb{B}_{-}(D/S)$ by \cite[Lemma 2.9]{Lai11}. Choosing an ample divisor $A$ on $X$, we may run a $(K_X + \Delta + \epsilon D)$-MMP with scaling of $A$ over $S$ that contracts all the prime components of $\mathbb{B}_{-}(D/S)$. Therefore $D$ is $f$-exceptional. Moreover, because contracting $D$ decreases the Picard rank, there are only finitely many degenerate divisors. Conversely, suppose that $D$ is $f$-exceptional and maps to a divisor in $S$ via $f$. There exists a marked SQM $\alpha_0 :\: X \dashrightarrow X_0$ over $S$ and a divisorial contraction $X_0 \to X_0'$ over $S$ contracting $D_0:=\alpha_{0*} D$. Suppose for contradiction that $D$ does contain the fiber of the morphism $f^{-1}f(D) \to f(D)$ over the generic point of $f(D)$. Because $\alpha_0$ is small, $D_0$ contains the fiber of $f_0^{-1}f_0(D) \to f_0(D)$. But then contracting $D_0$ over $S$ would result in a contradiction to the upper semi-continuity of the dimension of the fibers of $f_0'$. 
\end{proof}

The following lemma is a generalization of 
\cite[Lemma 3.2]{Kaw97} to higher dimensions:

\begin{lemma}
\label{lem:excgen}
    Let $f\colon (X,\Delta)\to S$ be a klt Calabi--Yau pair.
    The space $V(X/S)$ is $\mathbb{R}$-linearly spanned by the classes of $f$-exceptional divisors.
\end{lemma}

\begin{proof}
    For $D$ an $f$-vertical prime divisor, we may find a Cartier divisor $D'$ containing $f(D)$. We have that $f^*(D')=mD + D''$ for some positive rational number $m$ and $D''$ an $f$-effective $\mathbb{Q}$-linear combination of divisors of insufficient fiber type. By \cref{lem:degenerate}, the divisor $D''$ is an $f$-effective $\qq$-linear combination of $f$-exceptional divisors. This shows that $-[D]\in B^e(X/S)$ and that $[D]=-\frac{1}{m}[D'']$ in the $\mathbb{R}$-span of $f$-exceptional divisors. 
\end{proof}

\subsection{Good fibrations} 
In this subsection, we define the notion of a \emph{good fibration} and prove some results about them.

\begin{definition}
    A fibration $f\colon X\rightarrow T$ is said 
    to be a \emph{good fibration} if every divisor on $X$ dominates a divisor on $T$ and the variety $T$ is $\mathbb{Q}$-factorial.
\end{definition}

A good fibration has the property that the pullback of a movable divisor is a movable divisor. Likewise, if the fibration $f\colon X\rightarrow T$ is defined over a variety $S$, then the pullback of a divisor on $T$ movable over $S$ is movable on $X$ over $S$.

The following lemma generalizes \cite[Lemma 3.3(2)]{Kaw97} and \cite[Lemma 2.16]{FHS21}. The proof from \cite[Lemma 2.16]{FHS21} can be taken verbatim, noting that the proof of \cite[Proposition 2.9]{Fil23} holds for $(X,\Delta)$ klt. The lemma says that any klt CY pair factors through a good fibration up to to a marked SQM. 

\begin{lemma}
\label{lem:nicefactorization}
    Assume the Abundance Conjecture in relative dimension $n$. Let $f\colon (X,\Delta)\rightarrow S$ be a klt Calabi--Yau pair of relative dimension $n$.
    There is a marked SQM $X\dashrightarrow X_0$ over $S$
    with structure morphism $f_0\colon X_0\rightarrow S$
    together with a factorization
    \begin{center}
    \begin{tikzcd}
        X_0 \arrow[r, "g_0"] \arrow[rr, bend left=30, "f_0"] 
        & S_0 \arrow[r, "h_0"] & S
    \end{tikzcd}
    \end{center} such that $h_0$ is birational and $g_0$ is a good fibration.
\end{lemma}

From the following lemma, it follows that we may lift pseudo-automorphisms accross good fibrations. 

\begin{lemma}
\label{lem:commutativity}
    For $(X/S,\Delta)$ a klt CY pair, suppose we have a diagram 
    \begin{center}
    \begin{large}
    \begin{tikzcd}
        X_1 \arrow[d, "f_1", swap] 
        & X_2 \arrow[d, "f_2"] \\ 
        T_1 \arrow [r, "\pi_T", swap, dashed]
        & T_2
    \end{tikzcd}
    \end{large}
    \end{center}
    where $\alpha_i \colon X \dashrightarrow X_i$ are marked SQMs of $(X/S,\Delta)$, $f_1$ is a good fibration over $S$, $f_2$ a projective contraction over $S$, and $\pi_T$ is a birational contraction over $S$. Then, we may find a sequence of flops $\pi_X \colon X_1 \dashrightarrow X_2$ making the diagram commute.
    
\end{lemma}

\begin{proof}
    Choose a divisor $A$ of $T_2$ ample over $S$. The divisor $f_1^* \pi_T^* A$ is movable over $S$, since $f_1$ maps prime vertical divisors to prime divisors. Run a $(K_{X_1}+\Delta_1 + \epsilon f_1^* \pi_T^* A)$-MMP over $S$, terminating with a good minimal model $X'$ whose ample model is $T_2$. Because $f_1^* \pi_T^* A$ is movable over $S$, this MMP $\psi$ is just a sequence of flops over $S$. We have a diagram

    \begin{center}
    \begin{large}
    \begin{tikzcd}
        X_1 \arrow[r, dashed, "\psi"] \arrow[d, "f_1", swap] 
        & X' \arrow[r, dashed, "\psi'"] \arrow[d, "f'"] & X_2 \arrow[dl, "f_2"]\\ 
        T_1 \arrow [r, "\pi_T", swap, dashed]
        & T_2
    \end{tikzcd}
    \end{large}
    \end{center}
    where $\psi':=\psi\circ\alpha_1^{-1}\circ\alpha_2$. 
    The divisor $f_1^* \pi_T^* A$ is semiample on a nonempty open subset of $X_1$ and $\psi$ is an isomorphism over a nonempty open subset of $T_2$, so the square is commutative. However, a priori, we do not know if the triangle on the right commutes. In order to replace $\psi'$ with a sequence of flops making this triangle commute, choose a divisor $A_X$ on $X_2$ ample over $S$. Running an $(K_{X'} + \Delta' + \epsilon \psi'^*A_X)$-MMP over $T_2$ yields such a sequence of flops.
\end{proof}

We immediately get the following corollary that we may lift pseudo-automorphisms across good fibrations:

\begin{corollary}
\label{cor:lifting}
    Suppose that $(X/S, \Delta)$ is a klt CY pair and $f:\: (X,\Delta)\to T$ is a good fibration over $S$. Using the canonical bundle formula, choose a divisor $\Delta_T$ such that $K_X + \Delta = f^*(K_T + \Delta_T)$ and $(T/S,\Delta_T)$ is klt CY. Then, given $\phi \in \PsAut(T/S,\Delta_T)$ we may find a $\tilde{\phi}\in \PsAut(X/S,\Delta)$ lifting $\phi$, i.e., such that the diagram
    \begin{center}
    \begin{large}
    \begin{tikzcd}
        X \arrow[d, "f", swap] \arrow[r, "\tilde{\phi}", dashed]
        & X \arrow[d, "f"] \\ 
        T \arrow [r, "\phi", swap, dashed]
        & T
    \end{tikzcd}
    \end{large}
    \end{center}
    is commutative.
\end{corollary}

\subsection{The cone conjecture in relative dimension 1}\label{subsec:rel-dim-1}

In \cite[Theorem 3.4]{FHS21}, the authors prove the geometric cone conjecture for fibrations $f:\: X\to S$ with $X$ terminal, $K_X\equiv_S 0$, and $f$ of relative dimension one. In this subsection, we extend their result to klt CY pairs $(X/S,\Delta)$ of relative dimension one. The following lemma is well-known to the experts. 

\begin{lemma}
\label{lem:fanotype}
    If $(X/S, \Delta)$ is a klt CY pair with $\Delta$ big over $S$, then $X$ is of Fano type over $S$.
\end{lemma}

\begin{proof}
    Since $\Delta$ is big over $S$ we may find $A$ relatively ample $\qq$-Cartier and $E$ effective over $S$ such that $\Delta\sim_{S,\qq} A+E$ (see, e.g.,~\cite[2.2.7]{Laz04a}).

    Since $K_X + \Delta \sim_{S,\qq} 0$, we have that 
    \begin{center}
        $-(K_X +(1-\epsilon)\Delta +\epsilon E) \sim_{S,\qq} \epsilon A$
    \end{center}
    implying that $(X/S, (1-\epsilon)\Delta + \epsilon E)$ is a Fano pair. Note that for $\epsilon>0$ small enough, the pair
    $(X,(1-\epsilon)\Delta+\epsilon E)$ has klt singularities
    as the klt condition is open (see, e.g.,~\cite{KM98}).
\end{proof}

\begin{corollary}
    If $(X/T,\Delta)$ is a klt CY pair with relative dimension $1$ and $\Delta^{\hor}\neq 0$, then $X$ is of Fano type over $T$.
    
\end{corollary}

\begin{proof}
    Since $(X/S,\Delta)$ is klt CY, $(X_\eta,\Delta_\eta)$ is as well. Hence, $X_\eta$ is either an elliptic curve or $\mathbb{P}^1$. As $\Delta_\eta\neq 0$, $X_\eta\cong \mathbb{P}^1$. By adjunction, we have that $\Delta\cdot F=2$ for $F$ a general fiber, hence $\Delta|_F \sim_\mathbb{Q} -K_F \sim c_1 (\mathcal{O}_{\mathbb{P}^1}(-2))$ implying that $\Delta$ is big over $S$. Therefore, the result follows by the previous lemma.
\end{proof}

\begin{theorem}
\label{thm:reldim1}
    Let $f:\: (X,\Delta) \to T$ be a klt CY pair of relative dimension one.
    Then, the cone conjecture holds for $f:\: (X,\Delta) \to T$.
\end{theorem}

\begin{proof}
    It suffices to prove the statement for $\Delta=0$ because if $\Delta_{\hor}\neq0$, then $X$ is of Fano type over $T$, and if $\Delta$ is purely vertical, then $\PsAut(X/T,\Delta)\subset \PsAut(X/T)$ is a finite index subgroup.  Let $\phi:\: X'\to X$ be a terminalization of $X$ over $T$ so that $X\to T$ is a terminal Calabi--Yau fiber space. Then, by \cite[Theorem 3.4]{FHS21}, the conjecture is true for $X\to T$. All exceptional divisors of $\phi$ are vertical over $T$, so there exists a finite index subgroup of $\PsAut(X'/T)$ consisting of pseudo-automorphisms of $X'/T$ that descend to $X/T$. 
\end{proof}

\begin{corollary}
\label{cor:finite}
    In the setting of the previous theorem, if $\rho(X/S)=v(X/S)+1$, then $\mathcal{A}^e(X/S)$ and $\mathcal{M}^e(X/S)$ are rational polyhedral cones.
\end{corollary}

\begin{proof}
By \cite[Lemma 3.3]{FHS21}, $\PsAut(X'/T)$ is finite, and hence so are $\PsAut(X/T,\Delta)$ and $\Aut(X/T,\Delta)$. By \cref{thm:reldim1}, there exist only finitely many chambers of $\mathcal{M}^e(X/T)$ and finitely many faces in each chamber. 
\end{proof}

\section{The cone conjecture for the generic fiber}\label{section4}

In this section, we work over a field $k$ of characteristic $0$ that is not necessarily algebraically closed. We adapt the arguments of Totaro in \cite{Tot10} to extend his result on the existence of a rational polyhedral fundamental domain for klt CY pairs in dimension two to this setting. In particular, for a klt CY pair $(X/S,\Delta)$ of relative dimension two, this proves the cone conjecture for $(X_\eta,\Delta_\eta)$. Recall that for klt CY pairs in dimension two, the movable cone and nef cones coincide and $\PsAut(X,\Delta) = \Aut(X,\Delta)$. 

Over a non-closed field $k$ of characteristic $0$, we have that 
\[
N^1(X) = N^1(X_{\bar{k}})^{\Gal(\bar{k}/k)}, 
\quad 
\mathcal{A}^e(X) = \mathcal{A}^e(X_{\bar{k}})^{\Gal(\bar{k}/k)},
\text{ and }
\NE(X) = \NE(X_{\bar{k}})^{\Gal(\bar{k}/k)}.
\]
Even though $\Gal(\bar{k}/k)$ is profinite, a simple topological argument shows that its action on $N^1(X_{\bar{k}})$ must factor through a finite group. A $\Gal(\bar{k}/k)$-orbit of a $(K_X + \Delta)$-negative extremal ray of $\NE(X_{\bar{k}})$ is a $(K_X + \Delta)$-negative extremal face of $\NE(X_{\bar{k}})$, since the divisor $K_X+\Delta$ is $\Gal(\bar{k}/k)$-invariant. The cone theorem thus works in the same way as it does in the algebraically closed case, except that when we produce contractions defined over $k$, they must contract $\Gal(\bar{k}/k)$-invariant faces of $\NE(X_{\bar{k}})$.

For smooth surfaces, an extremal $K_X$-negative ray of $N^1(X)$ corresponds to a $1$-cycle $C=\sum C_i$ on $X_{\bar{k}}$ where $C_i$ are in a Galois orbit. If $C^2<0$, we see that $C_i$ must be mutually disjoint $(-1)$-curves. In this case, the associated contraction blows down each $C_i$ to a smooth point. If $C^2=0$, any connected component of $C$ is either irreducible or a union of two transversely intersecting $(-1)$-curves. In this case, the associated contraction turns $X$ into a conic bundle over a smooth curve. If $C^2>0$, then $\rho^{G}(X) = 1$ and $X$ is a del Pezzo surface (see, e.g.~\cite[Example 2.18]{KM98}).

The following is a classical result in the theory of surfaces.

\begin{theorem}
    For $S$ a surface, the intersection pairing on $N^1(S)$ is non-degenerate with signature $(1,\rho(S)-1)$.
\end{theorem}

The Hodge index theorem is valid over non-closed fields of characteristic $0$. Indeed, since $X$ is projective over $k$, the Galois action will fix the class of an ample divisor. Hence, the restriction of the intersection pairing on $N^1(X_{\bar{k}})$ to $N^1(X)$ contains a class with positive self-intersection. Thus, we have the same result on $N^1(X)$.

In the coming arguments, we will also make use of the Zariski Decomposition \cite{Nak04}.

\begin{theorem}
    Let $D$ be an effective $\mathbb{Q}$-divisor on a smooth projective surface $X$. Then there exist uniquely determined effective $\mathbb{Q}$-divisors $P$ and $N$ with
    $D = P + N$
    such that
    \begin{enumerate}
        \item $P$ is nef;
        \item $N$ is zero or has a negative definite intersection matrix; and 
        \item $P\cdot C=0$ for every irreducible component $C$ of $N$.
    \end{enumerate}
\end{theorem}

Since $P$ and $N$ are intrinsically defined for the divisor $D$ defined over $k$, $P$ and $N$ are fixed under the Galois action. Hence, the Zariski Decomposition is also valid over non-closed fields of characteristic $0$.

\begin{lemma}
\label{lem:semiample}
    Let $(X,\Delta)$ be a klt CY pair of dimension 2. Then any nef effective $\mathbb{R}$-divisor on $X$ is semiample. 
\end{lemma}

\begin{proof}
    Since $(X,\Delta)$ is a klt pair and $K_X + \Delta \equiv 0$ is numerically trivially and hence nef, $K_X + \Delta$ is semiample by applying the abundance theorem in dimension two. This follows by applying abundance on $X_{\bar{k}}$ and using that $\mathcal{A}(X) = \mathcal{A}(X_{\bar{k}})^{\Gal(\bar{k}/k)}$. 
    Therefore we know that $K_X+\Delta$ is $\mathbb{R}$-linearly equivalent to zero. Now, for any nef effective $\mathbb{R}$-divisor $D$, we have that $(X, \Delta + \epsilon D)$ is klt for $\epsilon>0$ sufficiently small, and $K_X + \Delta + \epsilon D$ is nef. Hence, applying abundance once more tells us that $K_X + \Delta + \epsilon D \sim_{\mathbb{R}} \epsilon D$ is semiample. 
\end{proof}

We will also make use of the following lemma:

\begin{lemma}
\label{lem:numtriv}
Let $\pi:\: Y\to X$ be a birational morphism of $\qq$-factorial surfaces and $D$ a numerically trivial Cartier divisor on $Y$, then $\pi_* D$ is numerically trivial as well. 
\end{lemma}

\begin{proof}
    We may assume that $\Exc(\pi)$ is an irreducible curve $C$. The divisor $D- \pi^* \pi_* D$ is supported on $C$ and therefore its numerical class is the same as $\alpha C$ for some $\alpha \in \qq$. We have that $0 = (D- \pi^* \pi_* D) \cdot C = \alpha (C\cdot C)$. If $\alpha \neq 0$, this contradicts the Hodge Index theorem, so $\alpha=0$, $D = \pi^* \pi_* D$, and $\pi^* \pi_* D$ is numerically trivial. By the projection formula, so is $\pi_* D$.  
\end{proof}

By~\cite[Remark 2.2]{Kaw97}, we have the following theorem.

\begin{theorem}
    \label{thm:smoothsurfaces}
    The cone conjecture holds for smooth Calabi--Yau surfaces in characteristic $0$. 
\end{theorem}

We now extend this result to CY surfaces with klt singularities, adapting \cite[Theorem 3.3]{Tot10}. 

\begin{theorem}
    \label{thm:kltsurfaces}
    The cone conjecture holds for klt Calabi--Yau surfaces in characteristic $0$. 
\end{theorem}

\begin{proof}
    Let $Y$ be a klt CY surface over a field $k$ of characteristic $0$. Let $I$ be the global index of $Y$ so that $IK_Y$ is Cartier and linearly equivalent to zero. Letting $Z\to Y$ be the index one cover. The surface $Y$ has trivial canonical bundle, du Val singularities, and $Y\cong Z/(\mathbb{Z}/I)$. Let $M$ be the minimal resolution of $Z$. Since $M$ is smooth with $K_M\sim 0$, we have the result for $M$ by \cref{thm:smoothsurfaces}. That is, there exists a rational polyhedral fundamental domain for $\Aut(M)\curvearrowright \mathcal{A}^e(M)$. Since the minimal resolution is unique for surfaces, $\mathbb{Z}/I$ acts on $M$ and moreover on $M_{\bar{k}}$. Define 
    \[
    X = M/(\mathbb{Z}/I) = M_{\bar{k}}/(\mathbb{Z}/I\times \Gal(\bar{k}/k)).
    \]
    The proof of \cite[Corollary 1.9]{OS01}, an equivariant cone conjecture for finite $G$-actions on surfaces by Oguiso and Sakurai, is valid for $G$ a subgroup of $\Aut(M_{\bar{k}}) \times \Gal(\bar{k}/k)$ whose projection to the first factor is finite. As $N^1(M_{\bar{k}})$ is finite dimensional, the action  $\Gal(\bar{k}/k)\curvearrowright N^1(M_{\bar{k}})$ factors through a finite group. Applying this result yields the cone conjecture for $\Aut(X)\curvearrowright \mathcal{A}^e(X)$. The result then follows from \cite[Lemma 3.4]{Tot10}, whose proof is valid over fields of characteristic zero. 
\end{proof}

The arguments in \cite[Theorem 4.1]{Tot10}, extending the cone conjecture from klt CY surfaces to klt pairs of dimension two, are valid over non-closed fields of characteristic $0$. We restate the arguments here for the sake of completeness.

By the Hodge index theorem, the intersection product on $N^1(X)$ has signature $(1,\rho(X)-1)$. This allows us to identify the positive cone 
\[
\mathcal{C}_+ := \{[z]\in N^1(X) \mid  z^2>0 \text{ and } z\cdot A >0\}\footnote{$A$ being an ample divisor.}
\]
modulo scalars with a hyperbolic $(\rho(X)-1)$-space, and the nef cone with a convex subset of the hyperbolic space. The coming arguments make use of the following result of Looijenga to produce a rational polyhedral fundamental domain: 

\begin{theorem}[{Cf.~\cite[Proposition 4.1, Application 4.14]{Loo14}}]
    \label{thm:ddomain}
    Let $S$ be a finitely generated free $\mathbb{Z}$-module and $A$ a closed strictly convex cone in $S_{\mathbb{R}}$ with nonempty interior. Let $G$ be a subgroup of $\text{GL}(S)$ which preserves the cone $A$. Suppose that there is a rational polyhedral cone $C \subset A$ such that $\bigcup_{g \in G} gC$ contains the interior of $A$. Then $\bigcup_{g \in G} gC$ is equal to the convex hull $A^+$ of the rational points in $A$, and $G$ has a rational polyhedral fundamental domain on $A^+$. That is, there exists a rational polyhedral cone $\Pi$ such that $\bigcup_{g \in G} g\Pi = A^+$ and $\interior(\Pi) \cap g(\interior(\Pi)) = \emptyset$ unless $g = 1$.
\end{theorem}

This lemma is fundamental to prove the following, which is \cite[Theorem 4.1]{Tot10} without the algebraically closed assumption:

\begin{theorem}
    \label{thm:kltpairs}
    Let $(X,\Delta)$ be a klt CY pair of dimension two over a field $k$ of characteristic $0$. Then the cone conjecture holds for $(X,\Delta)$.
\end{theorem}

\begin{proof}
    By \cite[Lemma 3.4]{Tot10}, whose proof is valid over fields of characteristic $0$, it suffices to prove the result for the terminalization of $(X,\Delta)$, and hence we may assume that $(X,\Delta)$ is terminal and, in particular, $X$ is a smooth surface. We may also assume that $\Delta\neq0$, because otherwise we may apply \cref{thm:kltsurfaces}, and that $\rho(X)>2$. Nikulin showed that, in this case, $X$ is geometrically rational (see~\cite[Lemma 1.4]{AM04}).\\ 
    
    \noindent\textit{Claim 1:} If $X$ is a geometrically rational surface over a field of characteristic $0$ and with $-nK_X$ effective for some $n>0$, every nef line bundle on $X$ is effective.\\

    \begin{proof}[Proof of Claim 1]
        Let $L$ be nef on $X$. We estimate $h^0(X,L)$ using Riemann-Roch for surfaces:
        \begin{align*}
            \chi(X,L) &= \chi(X,\mathcal{O}_X) + \frac{1}{2}(L^2 - K_X\cdot L).
        \end{align*}
        Since $X$ is geometrically rational, we have $\chi(X,\mathcal{O}_X)=1$. Moreover, since $L$ is nef and a multiple of $-K_X$ is effective, $\dfrac{1}{2}(L^2 - K_X\cdot L)\geq 0$, implying that $\chi(X,L)\geq 1$. Moreover, we have $A\cdot(K_X - L) < 0$ for any ample line bundle $A$, implying that $h^0(X, K_X - L) = 0$. Hence,
        \begin{align*}
            h^0(X,L) &= h^0(X,L) + h^0(X, K_X - L) \geq \chi(X,L) \geq 1.
        \end{align*}
    \end{proof}

    The cone conjecture is trivial if $\rho(X) = 2$, so we may assume that $\rho(X)\geq 3$, and hence that every $K_X$-negative extremal ray in $\overline{NE}(X)$ is spanned by the class of a $(-1)$-curve.

    Since $K_X + \Delta \equiv 0$, the divisor $-K_X$ is effective and possesses a Zariski Decomposition $-K_X = P + N$ with $P$ a nef $\mathbb{Q}$-divisor, $N$ an effective $\mathbb{Q}$-divisor with negative definite intersection pairing, and $P\cdot N = 0$. In particular, the divisor $P := \Delta - N$ is effective and therefore semiample by \cref{lem:semiample}. 

    The Iitaka dimension of $P$ is either $0$, $1$, or $2$, each case for which we treat separately. For $\kappa(P) = 0$ or $1$, we will show that $\mathcal{A}^e(X)$ is the union of a collection of rational polyhedral cones that fall into finitely many $\Aut(X,\Delta)$-orbits and thus deduce the existence of a rational polyhedral fundamental domain by \cref{thm:ddomain}.\\ 

\noindent \textit{Case 1:} The Iitaka dimension of the divisor $P$ equals $2$.\\ 

    In this case, $-K_X$ is big and therefore by \cref{lem:fanotype} $X$ is of Fano type. Hence, the cone theorem applies in this case and $\mathcal{A}^e(X)$ is itself rational polyhedral.\\

\noindent \textit{Case 2:} The Iitaka dimension of the divisor $P$ equals $1$.\\

    In this case, the semiample divisor $P$ determines a fibration to a curve $X\to B$ and $P^2 = 0$. We have that $-K_X \cdot P = (P + N)\cdot P = 0$ so, by adjunction, the generic fiber is genus $1$. We may contract Galois-orbits of $(-1)$-curves contained in the fibers of the induced map $X_{\bar{k}}\to B_{\bar{k}}$ to obtain a factorization $X\xrightarrow{\pi} Y \to B$ through a minimal fibered surface $Y\to B$. Writing $\Delta_Y := \pi_*\Delta$, we have that $K_X + \Delta = \pi^*(K_Y + \Delta_Y)$ and hence $(Y, \Delta_Y)$ is a klt CY pair with $(X,\Delta)$ being the terminalization of $(Y, \Delta_Y)$. 

    Since $\Delta_Y \equiv -K_Y$ has degree $0$ on all curves contracted by $Y\to B$, applying Zariski's lemma \cite[Lemma 8.2]{Barth04} implies that $\Delta$ is a positive sum of multiples of fibers of $X\to B$. Since $Y\to B$ is minimal, the Mordell-Weil group $G:=\Pic^0(X_\eta)$ acts by automorphisms of $Y$ over $B$, and hence automorphisms in $\Aut(Y/B,\Delta_Y)$. Since terminalizations are unique in dimension two, the Mordell-Weil group acts by automorphisms in $\Aut(X/B,\Delta)$ on $X$.\\   

\noindent\textit{Claim 2:} There are finitely many $G$-orbits, and therefore finitely many $\Aut(X,\Delta)$-orbits, of $(-1)$-curves in $X$.\\
\begin{proof}[Proof of Claim 2]
    First, for each curve $C$ contracted by $X\to B$, the intersection number of any $(-1)$-curve $E$ with $C$ is bounded, independent of $E$. Indeed, by adjunction, $1 = K_X \cdot E = (P + N)\cdot E$. So, if $E$ is not one of the finitely many curves in $N$, we have that $N\cdot E\geq0$ and therefore $P\cdot E\leq 1$. A fiber of $X\to B$ is numerically equivalent to a fixed multiple of $P$. Writing the numerical class of a fiber as an effective linear combination of $C$ and the other curves in the same fiber of $X\to B$ gives a bound on $C\cdot E$ for $C$ any curve contracted by $X\to B$.

    We have that 
    $$\Pic(X_\eta) = \Pic(X) / (aP, C_1,\hdots, C_r)$$
    where $a>0$ and $C_1,\hdots, C_r$ the curves contained in reducible fibers of $X\to B$. Since the degree of a line bundle on $X$ on a general fiber of $X\to B$ is given by the intersection number of with $bP$ for some $b>0$, we have that the Mordell-Weil group $G:= \Pic^0(X_\eta)$ is the subquotient of $\Pic(X)$ given by
    $$G=P^\perp / (aP, C_1,\hdots, C_r).$$
    An element $x\in G$ acts by a translation on $X_\eta$, and extends to an automorphism of $(X,\Delta)$. If $y$ is an element of $\Pic(X_\eta)$, translation of $y$ by $x$ is given by 
    $$\varphi_x(y) = y + \deg(y)x.$$
    Since $bP$ is the class of a general fiber of $X\to B$, the element $\varphi_x$ acts on $\Pic(X)$ by 
    \[
    \varphi_x(y) = y + (y\cdot bP)x \;\;\; \mod aP, C_1,\hdots, C_r.
    \]
    By \cite[Proposition 8.12(iii)]{Ray70}, for an element $x\in P^\perp$ with $x\cdot C_i = 0$ for all of the curves $C_i$, the automorphism $\varphi_x$ of the minimal elliptic surface $Y$ acts on all singular fibers $F$ of $Y\to B$ by an automorphism in the connected component of the identity in $\Aut(F)$. This implies that the same must be true for the singular fibers of $X$ and, in particular, $\varphi_x$ induces the identity permutation on the curves in each fiber of $X\to B$. Since the action of $G$ on $\Pic(X)$ preserves the intersection product, $\varphi_x$ acts on $\Pic(X)$ by the strictly parabolic transformation 
    \[
    \varphi_x(y) = y + (y\cdot bP)x - \left[x\cdot y + \frac{1}{2}(x\cdot x)(y\cdot bP)\right](bP)
    \]
    for all $x\in P^\perp$ with $x\cdot C_i = 0$ for each curve $C_i$ and for all $y\in \Pic(X)$. 

    Now, to show that there are finitely many $G$-orbits of $(-1)$-curves, we show that a $(-1)$ curve $E$ not contained in a fiber of $X\to B$ has its $G$-orbit determined by its intersection numbers $m:=E\cdot bP$, $E\cdot C_i$ for each $C_i$, and its residue $E \mod m\Pic(X)$. Indeed, suppose that we have two $(-1)$-curves $E_1$ and $E_2$ such that $m:= E_1\cdot bP = E_2 \cdot bP$, $E_1 \cdot C_i = E_2 \cdot C_i$ for each $C_i$ and $E_1 \cong E_2 \mod m\Pic(X)$. Letting $x = (E_2 - E_1)/m\in \Pic(X)$, then we have that $x\in P^\perp$, $x\cdot C_i=0$ for each curve $C_i$, and 
    \[
    \varphi_x(E_1) = E_1 + (E_1\cdot bP)x - \left[(x\cdot E_1) + \frac{1}{2}(x\cdot x)(E_1\cdot bP)\right](bP) = E_2.
    \]
    Recall that there are only finitely many $(-1)$-curves contained in the fibers of $X\to B$. Moreover the $(-1)$-curves not contained in the fibers of $X\to B$ have their $G$-orbits determined by the finitely many possibilities of values for $m$, their intersection numbers $E\cdot C_i$, and residue $E \mod m\Pic(X)$. Therefore, there are only finitely many $G$-orbits for $(-1)$-curves on $X$.
    \end{proof}

    Now, we turn to describe all of the extremal rays of $\NE(X)$:
    Every $K_X$-negative extremal ray is spanned by a $(-1)$-curve, and every $K_X$-positive extremal ray is spanned by one of the finitely many curves in $N$. Suppose $\mathbb{R}^{\geq 0} x$ is an extremal ray of $\NE(X)$ inside of $K_X^\perp$, and that $x$ is not a multiple of a curve in $N$. Then, $x\cdot N\geq 0$. Since $P$ is nef and $x\cdot (-K_X) = x \cdot (P + N) = 0$, $x\cdot P = 0$ as well. Since $P^2 = 0$, by the Hodge index theorem either $x$ is a multiple of $P$ or $x^2 < 0$. In the latter case, the ray $\mathbb{R}^{\geq 0}$ is spanned by a curve $C$, and since $C\cdot P = 0$, $C$ is contained in a fiber of $X\to B$. Because there are only finitely many numerical classes of curves contained in fibers of $X\to B$, we have that all but finitely many extremal rays of $\NE(X)$ are spanned by $(-1)$-curves. 

    Furthermore, we show that the only possible limiting ray of the $(-1)$-rays is $\mathbb{R}^{\geq 0} P$. Indeed, suppose that $\mathbb{R}^{\geq 0} x$ is a limit ray of $(-1)$-rays $\mathbb{R}^{\geq 0} E_i$, for $i\in \mathbb{N}$. For an ample line bundle $A$, all $(-1)$-curves $E_i$ have $E_i^2 = -1$, $-K_X \cdot E_i = -1$, and for an infinite sequence of $(-1)$-curves, the degrees $A\cdot E_i$ must approach infinity, while $0 < x\cdot A < \infty$. Since $\lim_{i\to\infty} \frac{E_i}{E_i \cdot A} = \frac{x}{x \cdot A}$, we must have $x^2 = 0$ and $-K_X \cdot x = 0$. Also, because $N$ has non-negative intersection with co-finitely many $(-1)$-curves, we must have that $N\cdot x \geq 0$, implying that $P\cdot x = 0$ since $-K_X\cdot x = (P+N)\cdot x = 0$ and $P$ is nef. Since $x\cdot P=0$, $P^2 = 0$, and $x^2 = 0$, $x$ is a multiple of $P$ by the Hodge index theorem. 

    We now deduce that the nef cone $\overline{\mathcal{A}}(X)$ is rational polyhedral near any point $y$ not inside of the ray $\mathbb{R}^{\geq 0} \cdot P$, meaning that there exists a neighborhood $U$ of $y$ such that $\overline{\mathcal{A}}(X)\cap U$ is the intersection of a rational polyhedral cone with $U$. Such a point $y$ has $y^2\geq 0$ and $P\cdot y \geq 0$ by nefness and, in fact, $P \cdot y > 0$ since $P\cdot y = 0$ would imply that $y$ is a multiple of $P$. Because the only possible limiting ray is $\mathbb{R}^{\geq 0} P$, there is a neighborhood of $y$ which has positive intersection with all but finitely many $(-1)$-curves. Since all but finitely many extremal rays are spanned by $(-1)$-curves, there are only finitely possible extremal rays that may intersect trivially with a class in this neighorhood. Hence, $\overline{\mathcal{A}}(X)$ is rational polyhedral near $y$.

    For each $(-1)$-curve $E$ not contained in a fiber of $X\to B$, the face $\overline{\mathcal{A}}(X)\cap E^{\perp}$ is rational polyhedral because $E\cdot P > 0$, and hence the cone $\Pi_E$ spanned by $P$ and $\overline{\mathcal{A}}(X)\cap E^{\perp}$ is rational polyhedral. Let $x$ be a nef $\mathbb{R}$-divisor on $X$ that is not in the ray spanned by $P$ and let $c$ be the maximal real number such that $y:= x-cP$ is nef. Then, there exists an extremal ray $R$ such that $R\cdot y = 0$. Since $x$ and $y$ have the same intersection number against any curve contracted by $X\to B$, the ray $R$ is spanned by a curve $C$ not contained in a fiber of $X\to B$. As shown above, all such curves spanning extremal rays of $\NE(X)$ are $(-1)$-curves. This shows that $y \in E^{\perp}$ and moreover $x\in \Pi_E$. Any rational point $z\in \overline{\mathcal{A}}(X)$ is in $\mathcal{A}^e(X)$ because nef line bundles are effective on $X$, so the rational polyhedral cones $\Pi_E$ are contained in $\mathcal{A}^e(X)$. Thus, $\overline{\mathcal{A}}(X) = \mathcal{A}^e(X)$, as both are the union of the cones $\Pi_E$. Since there are only finitely many orbits of $\Aut(X,\Delta)$-orbits of $(-1)$-curves in $X$, the set of cones $\Pi_E$ fall into finitely many $\Aut(X,\Delta)$-orbits, implying the existence of a rational polyhedral fundamental domain for $\Aut(X,\Delta)\curvearrowright \mathcal{A}^e(X)$ by \cref{thm:ddomain}.\\

\noindent\textit{Case 3:} The Iitaka dimension of the divisor $P$ equals $0$.\\ 

    In this case, $P$ is numerically trivial, $-K_X \equiv N$, and moreover $N=\Delta$, as the unique effective $\mathbb{R}$-divisor numerically equivalent to $-K_X$. 

     Letting $N_1,\hdots, N_r$ be the irreducible components of $N$, by the negativity lemma, there is a positive linear combination $D = \sum a_i N_i$ with $D\cdot N_i = -1$ for each $i$. Moreover, the support of $D$ is a union of connected components of $N$. Since $(X,\Delta+\epsilon D)$ is klt for small $\epsilon$, we may run a $(K_X + \Delta + \epsilon D)$-MMP that contracts precisely $N$, since $(K_X + \Delta + \epsilon D )\cdot N_i = \epsilon D \cdot N_i < 0$. Denote $\pi:\: X\to Y$ the associated contraction. We have that $Y$ is a klt CY surface, $K_X + \Delta = \pi^*(K_Y)$ and $(X,\Delta)$ is the terminalization of $Y$. 

    We now describe all of the extremal rays of $\NE(X)$: Again, all $K_X$-negative extremal rays are spanned by $(-1)$-curves. The $K_X$-positive curves must be spanned by one of the finitely many curves in $N$. Suppose now that $\mathbb{R}^{\geq0}x$ is an extremal ray of in $K_X^{\perp}$. This ray may be spanned by one of the finitely many curves in $N$. If this is not the case, then $x\cdot N_i \geq0$ for each $i$. Since $-K_X \cdot x = 0$, we moreover have $x\cdot N_i = 0$, and therefore $x = \pi^*(w)$ for some $w\in \NE(Y)$. So an extremal ray in $K_X^{\perp}$ is spanned either by one of the curves $N_i$ or the pullback of a curve from $Y$. 

    Let $C$ be a $(-1)$-curve in $X$. We have that 
    $$
    1=-K_X\cdot C = \left(\sum a_i N_i \right)\cdot C = \sum a_i \lambda_i
    $$
    where $a_i$ are fixed positive numbers and $\lambda_i:= N_i\cdot C$. If $C$ is not among the curves contained in $N$, then $\lambda_i\geq0$, and there are only finitely many possibilities for the tuple of natural numbers $(\lambda_1,\hdots, \lambda_r)$, which we will call the \textit{type} of $C$. 

    Now, we describe the nef cone of $X$: A divisor class $u$ on $X$ may be written as $\pi^*(y) - \sum b_i N_i$ for some real numbers $b_i$ and $y\in N^1(Y)$. By push-pull, if $u$ is nef, $y$ must be as well. Moreover, $u$ must have non-negative intersection with each curve $N_i$. This property defines a half-space in $\bigoplus \mathbb{R}\cdot N_i$ for each curve $N_i$, restricting $(b_1,\hdots, b_r)$ to lie in a rational polyhedral cone $B$ defined by the intersection of these half-spaces. By the negativity lemma, the cone $B$ must be contained in $[0,\infty)^r$. Since the extremal rays of $\NE(X)$ are given by the curves $N_i$, curves pulled back from $Y$, and $(-1)$-curves, a class $u=\pi^*(y) - \sum b_i N_i$ is nef if and only if $y$ is nef on $Y$, $(b_1,\hdots, b_r)\in B$, and $u$ has non-negative degree on the $(-1)$-curves other than the curves $N_i$. For any $(-1)$-curve $C\neq N_i$, this last condition, by push-pull, tells us that $\sum \lambda_i b_i \leq y\cdot \pi_*(C)$.

    Since terminalizations are unique in dimension two automorphisms of $Y$ lift to automorphism of $(X\Delta)$ and moreover $\Aut(X,\Delta) = \Aut(Y)$. Since we know the cone conjecture for the klt CY surface $Y$, it suffices to prove that for any rational polyhedral cone $S\subset \mathcal{A}^e(Y)$, the inverse image under $\pi_*:\: \mathcal{A}^e(X) \to \mathcal{A}^e(Y)$ is rational polyhedral, allowing us to lift the fundamental domain. Since nef $\mathbb{Q}$-divisors on $X$ are effective, rational polyhedral cones in $\overline{\mathcal{A}}(X)$ are contained in $\mathcal{A}^e(X)$, so it suffices to show that the inverse image $T$ of $S$ inside of $\overline{\mathcal{A}}(X)$ is rational polyhedral. We may check this locally around integral points of $Y$.

    First consider an integral point $y_0\in S$ with $y_0^2>0$. It suffices to show that there are finitely many $(-1)$-curves needed to define $T$ over a neighborhood of $y_0\in S$. To do this, by our description of the nef cone, we may show that for each type $\lambda = (\lambda_1,\hdots, \lambda_r)$, there is a finite set $Q$ of $(-1)$-curves of type $\lambda$ such that for all $y$ in a neighborhood of $y_0$, $y\cdot \pi_*(C)$ is minimized among all $(-1)$-curves $C$ of type $\lambda$. We show that this is the case. The type $\lambda$ of $C$ determines the rational number $c:=\pi_*(C)^2$. The intersection pairing, by the Hodge index theorem, has signature $(1,\rho(Y)-1)$. Since $y_0>0$, the intersection of the hyperboloid $\{z\in N_1(Y) \mid z^2 = c\}$ with $\{z\in N_1(Y) \mid |z\cdot y_0|\leq M\}$ is compact for any $M$. Hence, there are finitely many integral classes $z\in N_1(Y)$ with $z^2=c$ and any given bounds on $z\cdot y_0$, and the same finiteness holds when we allow $y$ to vary in a neighborhood of $y_0$. 

    Now, consider an integral point $y_0\in S$ such that $y_0^2 =0$. Since $\mathcal{A}^e(Y)$ is contained in the positive cone 
    \[
    \{y\in N^1(Y) \mid y^2>0 \text{ and } A\cdot y > 0\}
    \]
    for $A$ ample on $Y$, $y_0$ spans an extremal ray of $\mathcal{A}^e(Y)$. By \cref{thm:abundance}, we have an associated fibration $Y\to L$ to a curve $L$. For a point $p\in Y$ over which $\pi:\: X\to Y$ is not an isomorphism, let $D$ be a curve through $p$ and contained in a fiber of $Y\to L$, and let $C$ be the proper transform of $D$ in $X$. The curve $C$ is then contained in a singular fiber of $X\to L$, $C$ has negative self intersection, and therefore $C$ spans an extremal ray of $\NE(X)$. Since $C$ is the proper transform of a curve containing a point over which $\pi$ is not an isomorphism, there is a $N_i$ that intersects $C$ positively, since $C$ is not among the $N_i$, our description of $\NE(X)$ tells us that $C$ must be a $(-1)$-curve. Each point $p$ as above corresponds to a connected component $R$ of $N$. Moreover, we have shown that for each connected component $R$ of $N$, there exists a $(-1)$-curve $C$ on $X$ such that $y_0\cdot \pi_*(X) = 0$ and $\lambda_i>0$ for some $i$ corresponding to an $N_i$ in $R$. 

    Since a $(-1)$-curve $C$ with $y_0\cdot \pi_*(C)=0$ must be contained one of the finitely many singular fibers of $X\to L$, we know the set $Q$ of such curves must be finite. We now prove that the curves in $Q$ are enough to define the extremal rays of $T$ over a neighborhood of $y_0\in S$. Up to scalars, we may view such a neighborhood as a set of linear combinations $y = y_0 + \sum c_i v_i$ for $v_i$ some nef classes on $Y$ and $c_i\geq0$ sufficiently small. From this description, it is evident that $y\cdot \pi_*(C) \geq y_0 \cdot \pi_*(C)$ for $y$ in this neighborhood and $C$ any $(-1)$-curve on $X$ apart from the curves $N_i$. Since $y_0$ is integral and nef, $y\cdot \pi_*(C)\geq 1$ for curves $(-1)$-curves $C$ outside of $Q$, and near $0$ for curves $C$ in $Q$. 
    
    The negativity lemma tells us that if $u = \pi^*(y) - \sum b_i N_i$ is a class in $N^1(X)$ with non-negative intersection with each $N_i$ (i.e., $(b_1,\hdots, b_r)\in B$), then each $b_i\geq 0$ and the support of $b$ is a union of connected components of $N$. Therefore, if $b_i=0$ for some $i$, then $b_j=0$ for all $N_j$ in the same connected component as $N_i$. Hence, for $(b_1,\hdots, b_r)\in B$, if $b_i$ is close to $0$ for some $i$, the same is true for all $b_j$ corresponding to $N_j$ in the same connected component of $N_i$.  Earlier, we showed that for each connected component $R$ of $N$, there exists a $(-1)$-curve $C$ in $Q$ such that $\lambda_j>0$ for $\lambda_j$ corresponding to some $N_j$ in $R$. For this $C$, the inequality $\sum \lambda_i b_i \leq y\cdot \pi_*(C)$ tells us that $b_j$ is close to $0$, and hence $b_i$ is close to $0$ for all $i$ corresponding to $N_i$ in the same connected component as $N_j$. Taking such a $C\in Q$ for each connected component $R$, it follows that $b_i$ is close to $0$ for all $i$. Since $y\cdot \pi_*(C)$ is at least $1$ for $C$ outside of $Q$, the inequalities $\sum \lambda_i b_i \leq y\cdot \pi_*(C)$ for the curves $C$ outside of $Q$ are implied by the same inequalities for $C$ inside of $Q$. Since $Q$ is finite this shows that $T$ is rational polyhedral over a neighborhood of $y_0$ in $S$.
\end{proof}

\section{Finiteness of fiber space structures}
\label{sec:finiteness-fiber-space}

In this section, we prove the following result
regarding finiteness of fiber spaces
for klt log Calabi--Yau pairs
of relative dimension two.

\begin{proposition}
\label{prop:FFSS}
    For $(X/S,\Delta)$ a klt CY pair of relative dimension two, there are finitely many Mori faces of $\mathcal{M}^e(X/S)$ corresponding to SQM fiber space structures up to $\PsAut(X/S,\Delta)$.
\end{proposition}

\begin{proof}
    A fiber space structure $X_0 \to T \to S$ on a marked SQM $\alpha_0:\: X\dashrightarrow X_0$ may be of relative dimension $1$ or $2$. By \cref{lem:crepant}, there are only finitely many SQM fiber space structures of relative dimension two. So, we only consider SQM fiber space structures of relative dimension one.
    
    By \cref{thm:kltpairs}, we have finitely many fiber space structures on the generic fiber $X_\eta$ up to the action of $\Aut(X_\eta, \Delta_\eta)$. Therefore, we know that there are finitely many SQM fiber space structures of $(X/S,\Delta)$ of relative dimension one up to birational equivalence. Here, we say that two SQM fiber space structures $\phi_1:\: X_1\to T_1$ and $\phi_2:\: X_2\to T_2$ of $(X/S,\Delta)$ are birationally equivalent if there exists a commutative diagram 

     \begin{center}
        \begin{tikzcd}
            X_1 \arrow[r, dashed] \arrow[d, "\phi_{1}", swap] 
            & X_2 \arrow[d, "\phi_{2}"] \\ 
            T_1 \arrow [r, swap, dashed]
            & T_2
        \end{tikzcd}
        \end{center}
     with the horizontal maps being birational. Fix $\mathcal{C}$ to be one of the finitely many birational equivalence classes relative dimension one SQM fiber space structures. We now prove that there are finitely many SQM fiber space structures in $\mathcal{C}$ up to $\PsAut(X/S,\Delta)$.\\
    
   \noindent \textit{Claim:} It suffices to find a finite collection of SQM fiber space structures $\phi_i:\: (X_i, \Delta_i) \to T_i$ over $S$ in $\mathcal{C}$ such that each $\phi_i$ is a good fibration and the following condition is satisfied. An arbitrary fiber space structure $\phi':\:(X',\Delta')\to T'$ in $\mathcal{C}$ fits into a commutative diagram over $S$:
      \begin{center}
        \begin{tikzcd}
            X_i \arrow[r, "\pi_X", dashed] \arrow[d, "\phi_{i}", swap] 
            & X' \arrow[d, "\phi'"] \\ 
            T_i \arrow [r, "\pi_T", swap, dashed]
            & T'
        \end{tikzcd}
        \end{center}
    where $\pi_X$ is a sequence of flops over $S$ and $\pi_T$ a birational contraction of $T_i$ over $S$.

    \begin{proof}[Proof of the Claim]
    Each $T_i$ comes naturally equipped by the canonical bundle formula with a boundary $\Delta_{T_i}$ that makes $(T_i/S, \Delta_{T_i})$ into a klt Calabi Yau pair of relative dimension one. Applying~\cref{thm:reldim1} to each $(T_i/S, \Delta_{T_i})$, it follows that there are finitely many SQM birational contractions of each $(T_i/S, \Delta_{T_i})$ up to $\PsAut(T_i/S,\Delta_{T_i})$. Since $\phi_i$ is a good fibration and pullbacks of divisors movable over $S$ are movable over $S$, an SQM birational contraction of $T_i/S$ canonically gives rise to a SQM fiber space structure of $X/S$. Since the pseudo-automorphisms $\PsAut(T_i/S,\Delta_{T_i})$ lift to $\PsAut(X_i/S,\alpha_{i*}\Delta)\cong \PsAut(X/S,\Delta_{X})$ by \cref{cor:lifting} there are finitely many Mori faces of $\mathcal{M}^e(X/S)$ corresponding to SQM birational contractions of $T_i/S$ up to $\PsAut(X/S,\Delta_{X})$. By the assumption in our claim, all SQM fiber space structures in $\mathcal{C}$ correspond to an SQM birational contractions of $T_i/S$ for one of the finitely many $i$, so there are finitely many SQM fiber space structures in $\mathcal{C}$ up to $\PsAut(X/S,\Delta_{X})$ altogether. 
\end{proof} 
We have reduced the proof of the proposition 
to Lemam~\ref{lem:fsslem} below.
\end{proof}

\begin{lemma}
    \label{lem:fsslem}
    In any birational equivalence class $\mathcal{C}$ of relative dimension one SQM fiber space structures of $(X/S,\Delta)$, we may find a finite set of SQM fiber space structures $\phi_{i}:\:(X_i,\Delta_i)\to T_i$ over $S$ with each $\phi_{i}$ a good fibration, satisfying the following. Given any SQM fiber space structure $\phi':\:(X_0,\Delta')\to T_0$ in $\mathcal{C}$ there is a commutative diagram 
    \begin{center}
    \begin{tikzcd}
        X_i \arrow[r, "\pi_X", dashed] \arrow[d, "\phi_{D_i}", swap] 
        & X' \arrow[d, "\phi'"] \\ 
        T_i \arrow [r, "\pi_T", swap, dashed]
        & T'
    \end{tikzcd}
    \end{center}
    over $S$ for some $i$, with $\pi_X$ a sequence of flops over $S$ and $\pi_T$ a birational contraction of $(T_i/S, \Delta_{T_i})$.
    
\end{lemma}

\begin{proof}
   Let $(T'/S, \Delta')$ be a target of a fiber space structure $X'\to T'$ in $\mathcal{C}$. All such targets in $\mathcal{C}$ are crepant birational equivalent to one another, so we may find a $\mathbb{Q}$-factorial terminal model $(\hat{T}/S, \Delta_{\hat{T}})$ that is log CY over $S$ and admits a birational contraction over $S$ to any one of them. By \cref{lem:degenerate} there are only finitely many prime divisors on $(\hat{T}/S,\Delta_{\hat{T}})$ that may be contracted over $S$. For a subset $\mathcal{D}$ of degenerate divisors on $(\hat{T}/S,\Delta_{\hat{T}})$, denote by $(T_{\mathcal{D}}/S,\Delta_{T_{\mathcal{D}}})$ the model on which we contract the degenerate divisors in $\mathcal{D}$. Any target $(T'/S,\Delta_{T'})$ of an SQM fiber space structure in $\mathcal{C}$ whose centers of $\mathcal{D}$ on $T'$ have codimension $\geq 2$ is connected to $(T_\mathcal{D}/S,\Delta_{T_\mathcal{D}})$ by flops. Thus, because the number of such subsets $\mathcal{D}$ is finite, we may form a finite collection of SQM fiber space structures $\hat{\phi_{i}}: (\hat{X}_i/S,\Delta_{\hat{X}_i}) \to \hat{T}_i$ in $\mathcal{C}$ 
    such that every target of an SQM fiber space structure in $\mathcal{C}$ is connected to a unique $(\hat{T}_i/S,\Delta_{\hat{T}_i})$ by flops. We factor each $\hat{\phi_{i}}:\: \hat{X}_i\to \hat{T}_i$ up to flops through a good fibration $\phi_i:\: X_i\to T_i$ as in \cref{lem:nicefactorization}. For any target $T'$ in $\mathcal{C}$, we thus have a composition $T_i\dashrightarrow \hat{T}_i\dashrightarrow T'$ where the first map is a birational contraction and the second is a sequence of flops. Since this composition does not extract any divisors, it is an SQM birational contraction by \cref{lem:birationalcontraction}. 
\end{proof}

\begin{remark}
Let $(X/S,\Delta)$ be a klt CY pair of relative dimension $n$.
Assume the finitess of fiber space structures $X_\eta \rightarrow T_\eta$ up to ${\rm Aut}(X_\eta,\Delta_\eta)$
and the geometric cone conjecture in relative dimension $<n$.
Then, the same arguments as in \cref{prop:FFSS} and \cref{lem:fsslem} show the finiteness of SQM fiber space structures of $(X/S,\Delta)$ up to $\PsAut(X/S,\Delta)$.
\end{remark}

\section{The geometric cone conjecture in relative dimension 2}

In this section, we prove the geometric cone conjecture in relative dimension two.
Throughout this section $(X/S,\Delta)$ is a klt CY pair of relative dimension two. If $\kappa(X_{\eta}, -K_{X_{\eta}})=2$ so that $\Delta$ is big over $S$, then $X$ is of Fano type over $S$  by \cref{lem:fanotype} and the cone conjecture holds because $\mathcal{A}^e(X)$ is rational polyhedral by the cone theorem. Because of this, we will from now on assume that $\kappa(X_{\eta}, -K_{X_{\eta}})$ is either $0$ or $1$.  

\begin{lemma}
The restriction homomorphism
$r\colon N^1(X/S)\rightarrow N^1(X_\eta)$
is surjective.
Furthermore, the following equalities hold
\[
r(\mathcal{M}^e(X/S))=\mathcal{A}^e(X_\eta), r^{-1}(\mathcal{A}^e(X_\eta))\cap \overline{\mathcal{M}}(X/S) = \mathcal{M}^e(X/S), \text{ and } r^{-1}(\mathcal{B}(X_\eta))=\mathcal{B}(X/S).
\]
\end{lemma}

\begin{proof}
    The first statement follows from that fact that we may extend divisors from $X_\eta$ to divisors on $X$ by taking the closure. The first two parts of the second statement follow from the fact that movable divisors on $X_\eta$ are nef and that the closure of a basepoint free divisor on $X_\eta$ is $f$-movable on $X$. The statement on the big cones follows from the definition of $f$-big divisors. 
\end{proof}

\begin{lemma}
\label{lem:k=1}
If $\kappa(X_{\eta}, -K_{X_{\eta}})=1$, then either $H^1(X_\eta, \mathcal{O}_{X_\eta})=0$ or $X_{\bar{\eta}}$ is birational to $(E\times \mathbb{P}^1) / G$, an isotrivial elliptic fibration over $\mathbb{P}^1$ for some elliptic curve $E$ and finite group $G$. 
\end{lemma}

\begin{proof}
The cohomology groups $H^i(X,\mathcal{O}_X)$ are birational invariants of varieties with klt singularities. This follows from the Leray spectral sequence and the fact that klt singularities are rational.

If $X_{\bar{\eta}}$ is a smooth minimal elliptic surface over $\mathbb{P}^1$, then either $H^1(X_\eta, \mathcal{O}_{X_\eta})=0$ or $X_{\bar{\eta}} \cong (E\times \mathbb{P}^1) / G$ for some finite group $G$, an isotrivial elliptic fibration over $\mathbb{P}^1$ (see, e.g.,~\cite{Serr96}). Thus, it suffices to show that $X_{\bar{\eta}}$ is birational to a smooth minimal elliptic surface over $\mathbb{P}^1$ or a surface with irregularity $0$. If $X_{\bar{\eta}}$ is smooth, we have the Zariski Decomposition $-K_{X_{\bar{\eta}}}=N+P$ on $X_{\bar{\eta}}$ where $N$ has negative definite intersection pairing among its components, $P$ is nef, and $P\cdot N = 0$. Moreover, the section rings $R(X,D)$ and $R(X,P)$ coincide. In particular, $\kappa(X,D)=\kappa(X,P)$. $N$ only depends on the numerical class of $-K_{X_{\bar{\eta}}}$, so $\Delta \leq N$ and so $P$ is effective. Applying \cref{lem:semiample} implies that $P$ is semiample and so it defines a morphism to a curve $\phi_P:\: X_{\bar{\eta}} \to C$. We have that $P^2=0$ and $-K_{X_{\bar{\eta}}}\cdot P = (P+N)\cdot P = 0$, so, by adjunction, the generic fiber is an elliptic curve. Using the canonical bundle formula for this elliptic surface, we see that $K_{X_{\bar{\eta}}}+\Delta_\eta \sim \phi_P^*(K_C + B_C + M_C)$ with $B_C$ nonzero effective, implying that $K_C$ is anti-ample since $K_{X_{\bar{\eta}}}+\Delta_\eta \equiv 0$. Therefore, $C\cong \pp^1$ and $X_{\bar{\eta}}$ is a smooth rational elliptic surface. Hence, we obtain a smooth minimal elliptic surface over $\mathbb{P}^1$ after contracting the $(-1)$ curves in contained in the fibers over $X_{\bar{\eta}}\to \mathbb{P}^1$.

 In general, with $(X_{\bar{\eta}},\Delta_{\bar{\eta}})$ klt, run a $(K_{X_{\bar{\eta}}} + (1+\epsilon)\Delta_{\bar{\eta}})$-MMP to obtain a morphism $\pi:\: X_{\bar{\eta}} \to X'_{\bar{\eta}}$ to a singular elliptic surface such that $(K_{X_{\bar{\eta}}} + (1+\epsilon)\Delta_{\bar{\eta}})\equiv -\epsilon\Delta_{\bar{\eta}} \equiv -\epsilon K_{X_{\bar{\eta}}}$ is semiample. The Iitaka dimension of the anticanonical only may increase during this MMP, so $\kappa(X'_{\bar{\eta}}, -K_{X'_{\bar{\eta}}}) \in \{1,2\}$. 

If $\kappa(X'_{\bar{\eta}}, -K_{X'_{\bar{\eta}}})=2$, then by \cref{lem:fanotype}, the variety $X'_\eta$ is of Fano type and therefore $H^1(X'_\eta, \mathcal{O}_{X'_\eta})=0$ by Kodaira vanishing. Hence, $H^1(X_\eta, \mathcal{O}_{X_\eta})=0$. 

If $\kappa(X'_{\bar{\eta}}, -K_{X'_{\bar{\eta}}})=1$, then write $C$ for the image of the anticanonical morphism. We may take a minimal resolution $X_{\bar{\eta}}'' \to X'_{\bar{\eta}}$ over $C$ to obtain a smooth (relatively) minimal elliptic surface over $C$. By the argument in the first paragraph, $C\cong \pp^1$ and either $H^1(X''_\eta, \mathcal{O}_{X''_\eta})=0$ or $X_{\bar{\eta}}''\cong (E\times \mathbb{P}^1)/G$.
\end{proof}

\begin{notation}
    By \cref{thm:kltpairs}, we have rational polyhedral fundamental domain $\Pi(X_\eta)$ for the action of $\Aut(X_\eta, \Delta_\eta)$ on $\mathcal{A}^e(X_\eta)$. We define $\Pi_1(X_\eta,\Delta_\eta)$ to be those classes in $\Pi(X_\eta)$ intersecting a fixed ample class $H\in\mathcal{A}(X_\eta)$ with degree $1$, $\tilde{\Pi}_1(X/S,\Delta) = r^{-1}\Pi_1(X_\eta)\cap \mathcal{M}^e(X/S)$ and similarly $N^1_1(X/S)=r^{-1}(N^1(X_\eta))$. 

We define $$W(X/S,\Delta):= \{[z]\in N^1_1(X/S)/V(X/S) \mid r(z)\in \Pi_1(X_\eta)\cap \mathcal{B}(X_\eta)\}$$
\end{notation}

\begin{lemma}
\label{lem:H1Otriv}
If the irregularity $H^1(X_\eta, \mathcal{O}_{X_\eta})=0$, then the kernel of the restriction map 
\[
r:\: N^1(X/S)\to N^1(X_\eta)
\]
coincides with the subspace $V(X/S)$. Furthermore, in this case we have $\rho(X/S) = \rho (X_\eta) + v(X/S)$ and $W(X/S,\Delta) \cong \Pi_1(X_\eta,\Delta_\eta) \cap \mathcal{B}(X_\eta)$.
\end{lemma}

\begin{proof}
    Looking at the long exact sequence associated to the exponential sequence, we have $\Pic(X)\hookrightarrow H^2(X,\zz)$. If $[D]$ is the class of a Cartier divisor $D\in \Pic(X)$ inside of $H^2(X,\zz)$ with $D_\eta\equiv 0$, then $[D_\eta]_\qq\in H^2(X,\qq)$ must be zero, since the pairing on $H^2(X,\qq)$ is given by the intersection paring of divisors, and hence $D_\eta$ is torsion in $\Pic(X_\eta)$ by the universal coefficient theorem. Thus, $nD_\eta = 0$ for some $n>0$, implying that $nD$ is vertical.  
\end{proof}

\begin{proposition}
\label{prop:k=0}
    If $\kappa(X_{\eta}, -K_{X_{\eta}}) \leq 1$, then 
    one of the following conditions is satisfied:
    \begin{enumerate}
    \item[(i)] either $H^1(X_\eta,\mathcal{O}_{X_\eta})=0$
    and the consequence of \cref{lem:H1Otriv} holds; or
    \item[(ii)]  the image of the representation $\sigma:\: \PsAut(X/S,\Delta) \to \GL(N^1(X/S),\mathbb{Z})$ contains a subgroup $G(X/S,\Delta)$ such that:
    \begin{enumerate}
        \item[(1)] $G(X/S,\Delta)$ acts trivially on $N^1(X_\eta)$ and $V(X/S)$;
        \item[(2)] $G(X/S,\Delta)$ is isomorphic to a free abelian group of rank $\rho(X/S) - \rho(X_\eta) - v(X/S)$; and
        \item[(3)] $G(X/S,\Delta)$ acts on the fibers of the projection $W(X/S,\Delta)\to\Pi_1(X_\eta,\Delta_\eta)\cap \mathcal{B}(X_\eta)$ properly discontinuously as a group of translations making $W(X/S,\Delta) / G(X/S,\Delta)$ into a real torus bundle over $\Pi_1(X_\eta,\Delta_\eta)\cap \mathcal{B}(X_\eta)$. In particular, $$W(X/S,\Delta)/G(X/S,\Delta)\to \Pi_1(X_\eta,\Delta_\eta)\cap \mathcal{B}(X_\eta)$$ is proper.
        \end{enumerate}
    \end{enumerate} 
\end{proposition}

\begin{proof}

We proceed in two cases depending on the Iitaka dimension of the anti-canonical divisor.\\

\noindent \textit{Case 1:} We assume that the anti-canonical divisor $-K_{X_\eta}$ has Kodaira dimension zero.\\ 

We will proceed in four subcases which are not mutually exclusive. In the first two subcases, we analyze different isomorphism classes of the geometric generic fiber $X_{\bar{\eta}}$ while in the last two subcases, we analyze whether the divisor $\Delta_{\eta}$ is trivial.\\ 

\noindent\textit{Case 1.1:} We assume that $X_{\bar{\eta}}$ is a K3 surface.\\

In this case, $H^1(X_\eta, \mathcal{O}_{X_\eta})\otimes_{\mathbb{C}(S)} \overline{\mathbb{C}(S)}\cong H^1(X_{\bar{\eta}}, \mathcal{O}_{X_{\bar{\eta}}})=\{0\}$ and we may apply \cref{lem:H1Otriv}.\\
            
\noindent\textit{Case 1.2:} We assume that $X_{\bar{\eta}}$ is an abelian surface.\\

In this case, we produce a subgroup $G(X/S,\Delta)\leqslant \sigma(\PsAut(X/S,\Delta))$ that acts properly discontinuously as a group of translations on the fibers of induced restriction map
$$r_V:\: N^1(X/S) / V(X/S) \to N^1(X_\eta)$$
over $\mathcal{A}(X_\eta) = \mathcal{A}^e(X_\eta)\cap \mathcal{B}(X_\eta)$.
Further, we will find $G(X/S,\Delta)$ so that the quotient $W(X/S,\Delta) / G(X/S,\Delta)$ is a real torus bundle over $\mathcal{A}(X_\eta)\cap \mathcal{B}(X_\eta)$. Fix a divisor $D_0$ on $X$ such that $D_{0,\eta}\in\mathcal{A}(X_\eta)$ and denote $F_{D_{0,\eta}}$ the fiber of $r$ over $D_{0,\eta}$.
            
If $X_{\bar{\eta}}$ is abelian, then
$X_\eta$ is a $\mathbb{C}(S)$-torsor under $A_\eta := \Alb(X_\eta)$ induced by the Galois-invariant addition on $X_{\bar{\eta}}\cong \Alb(X_{\bar{\eta}})$ (see, e.g.,~\cite{ACV22}). Indeed, given a point $p$ of $X_{\bar{\eta}}$, we may define an isomorphism $\phi_p:\: \Alb(X_{\bar{\eta}})\to X_{\bar{\eta}}$ sending $e$ to $p$. The action $\Alb(X_{\bar{\eta}})\curvearrowright X_{\bar{\eta}}$ is defined by $a\cdot x = \phi_p(a + \phi_p^{-1}(x))$ and is independent of $p$. This action is Galois invariant, because given an element $\zeta\in \Gal(\overline{\mathbb{C}(S)}/\mathbb{C}(S))$, we have that 
                \begin{align*}
                    (\zeta a)\cdot (\zeta x) &= \phi_p(\zeta a + \phi_p^{-1}(\zeta x))\\
                    &= \phi_p(\zeta a + \zeta \phi_{\zeta^{-1}p}^{-1}(x))\\
                    &= \zeta \phi_{\zeta^{-1}p}(a + \phi_{\zeta^{-1}p}^{-1}(x))\\
                    &= \zeta \phi_{p}(a + \phi_{p}^{-1}(x))
                \end{align*}
            
                so the action descends to $\Alb(X_{\eta})\curvearrowright X_{\eta}$.
    
                Define $G(X/S,\Delta)$ to be the image under $\sigma$ of the subgroup $M\leqslant \PsAut(X/S,\Delta)$ of pseudo-automorphisms that act as $A_\eta$-translations on $X_\eta$ and act trivially on $V(X/S)$. This guarantees (1). There is a finite index subgroup of $\mathbb{C}(S)$-rational points $x$ of $A_\eta$ that induce a translation $T_x$ of $X_\eta$ that may be extended to an element $\theta_x\in \PsAut(X/S,\Delta)$ acting trivially on $V(X/S)$.
                Hence, elements of $M$ may be identified with $\mathbb{C}(S)$-rational points of $A_\eta$. 
                The map $\sigma':\: M \to N^1(X/S)/V(X/S)$ defined by $\sigma'(\theta) = [\theta_* D_0 - D_0]$ has image contained in $\ker(r_V)$. The divisor class $D_{0,\eta}$ determines a degree $d$ polarization $\psi_{D_{0,\eta}}:\: A_\eta \to A_\eta^\vee \cong \Pic_0(X_\eta)$ given by $x\mapsto T_x D_{0,\eta} - D_{0,\eta}$ that only depends on the numerical class of $D_{0,\eta}$ \cite[Proposition 5.3]{LB92}. So, for $\theta\in M$ and any divisor $D$ on $X$ such that $[D]\in F_{D_{0,\eta}}$, we have a linear equivalence $\theta_* D_\eta - D_\eta \sim \theta_* D_{0,\eta} - D_{0,\eta}$ on $X_\eta$. Letting $\theta = \theta_x$ and $D_\eta = \theta_y D_{0,\eta}$ for $x,y \in A_\eta(\mathbb{C}(S))$, we have that the linear equivalence $$T_{x+y} D_{0,\eta} - D_{0,\eta} \sim (T_x D_{0,\eta} - D_{0,\eta}) + (T_y D_{0,\eta} - D_{0,\eta})$$ extends over an open set of $S$, i.e., $$(\theta_x\circ \theta_y)_* D_{0} - D_{0} - (\theta_x)_* D_{0} + D_{0} - (\theta_y)_* D_{0} + D_{0} \in V(X/S).$$ Therefore $\sigma'$ is a group homomorphism. As a subgroup of a vector space, the image of $\sigma'(M)$ is a free abelian group. The rank $\sigma(M)$ is the same as the rank of $\sigma'(M)$, and $\dim \ker(r_V) = \rho(X/S) - v(X/S) - \rho(X_\eta)$. Thus, to show (2) and (3) it suffices to show that $\sigma'(M)$ has full rank in $\ker(r_V)$. If $D$ again is a divisor on $X$ such that $[D]\in F_{D_{0,\eta}}$, then the sum of points in $\psi^{-1}_{D_{0,\eta}}(D_\eta - D_{0,\eta})$ determine a $\mathbb{C}(S)$-rational point $x$ such that $T_x D_{0,\eta} - D_{0,\eta} \sim d(D_\eta - D_{0,\eta})$. Because this linear equivalence extends over an open set of $S$, $\sigma'(\theta_x) = [d(D-D_0)]\in N^1(X/S)/V(X/S)$. Therefore, $\sigma'(M)$ is of full rank in $\ker(r_V)$ and (2) and (3) hold.\\

                \noindent\textit{Case 1.3:} We assume that $\Delta_{\bar{\eta}}=0$.\\ 

                Passing to the minimal $\tilde{X}_{\bar{\eta}}$ resolution of the index one cover of $X_{\bar{\eta}}$, we obtain a K3 or abelian surface. If $\tilde{X}_{\bar{\eta}}$ is K3, then $X_{\bar{\eta}}$ has irregularity $0$ and we're done. 
        
                Otherwise, assume now that the index one cover of $X_\eta$ is an abelian torsor. Taking $\tau:\: \tilde{X}_\eta \to X_\eta$ to be the index one cover of $X_\eta$, $\tau$ is a cyclic Galois covering, and $\tilde{X}_\eta$ is a $\mathbb{C}(S)$-torsor under $\tilde{A}_\eta = \Alb(\tilde{X}_\eta)$. For a generator $\zeta\in \Gal(\tilde{X}_\eta / X_\eta)$ and $D$ a divisor on $X$ such that $[D]\in F_{D_{0,\eta}}$, we have a $\mathbb{C}(S)$-rational point $x$ on $\tilde{A}_\eta$ such that $(\theta_x)_*\tau^* D_{0,\eta} - D_{0,\eta} \sim d(\tau^* D_\eta - \tau^* D_{0,\eta})$, as before. The generator $\zeta$ permutes the points in the fiber $\phi_{\tau^*D_{0,\eta}}^{-1} (\tau^* D_\eta - \tau^* D_{0,\eta})$, so $x$ is $\zeta$-invariant, $\theta_x \circ \zeta = \zeta \circ \theta_x$, and $\theta_x$ descends to an element of $\PsAut(X/S,\Delta)$. We may thus apply the same argument as in the abelian case, defining $M\leqslant \PsAut(X/S,\Delta)$ to be the pseudo-automorphisms that act by $\zeta$-invariant $\tilde{A}_\eta$-translations on $X_\eta$ and trivially on $V(X/S)$.\\

                \noindent\textit{Case 1.4:} We assume that $\Delta_{\bar{\eta}}\neq 0$.\\
            
                For $0<\epsilon\ll1$, the pair $(X, (1+\epsilon) \Delta^{\hor})$ is klt and $\qq$-linearly equivalent to $\epsilon \Delta^{\hor}$. Run a $(K_X + (1+\epsilon) \Delta^{\hor})$-MMP over $S$ to obtain a birational contraction $\phi:\:X\to X'$ that contracts exactly the components of $\Delta^{\hor}$. We obtain a klt CY pair $(X'/S,\Delta'=\phi_* \Delta)$ such that $\Delta'_\eta=0$. From the $\Delta_\eta = 0$ case, we obtain a subgroup $M'\leqslant  \PsAut(X'/S,\Delta')$ defined as before. Restricting to the generic fiber, we obtain a birational contraction $\phi_\eta:\: X_\eta\to X'_\eta$ that contracts precisely the components of $\Delta_\eta$. There exists a finite index subgroup $M\leqslant M'$ that fixes the valuations over $X'$ corresponding to the components of $\Delta^{\hor}$. Hence, the pseudo-automorphisms of $M$ lift to $\PsAut(X/S,\Delta)$. Defining $G(X/S,\Delta)$ to be the image of the liftings of elements of $M$ under $\sigma:\: \PsAut(X/S,\Delta) \to \GL(N^1(X/S,\zz))$, conditions (1) and (2) hold automatically. Moreover, (3) follows from the fact that $M\leqslant M'$ is finite index.\\

    \noindent\textit{Case 3:} We assume that the anti-canonical divisor $-K_{X_\eta}$ has Kodaira dimension one.\\ 
    
        By \cref{lem:k=1}, either $X_\eta$ has irregularity $0$ or $X_{\bar{\eta}}$ is birational to an isotrivial elliptic fibration over $\mathbb{P}^1$. Addressing the latter, if $X_{\bar{\eta}}\cong (E\times \mathbb{P}^1)/G$ for $E$ an elliptic curve, then we run the same argument that we did for $X_{\bar{\eta}}$ abelian, noting that $X_{\eta}$ still possesses a free action of $A_\eta:=\Alb(X_\eta)$. 
        
        Otherwise, as in the proof of \cref{lem:k=1}, run a $(K_{X_{\eta}} + (1+\epsilon)\Delta_\eta)$-MMP over $\mathbb{P}^1$ to obtain a singular elliptic surface $X'_{\eta} \to \mathbb{P}^1$ and let $X''_{\eta}$ be the minimal resolution over $\mathbb{P}^1$. We have that $X''_{\bar{\eta}}\cong (E\times \mathbb{P}^1)/G$, an isotrivial elliptic fibration over $\mathbb{P}^1$. Let $(\tilde{X}/S,\tilde{\Delta})\to X$ be the terminalization of $(X/S,\Delta)$. We have a commutative diagram over $\mathbb{P}^1$:

        \begin{center}
        \begin{tikzcd}
            X_\eta \arrow[d, swap] 
            & \tilde{X}_\eta \arrow[l] \arrow[d] \\ 
            X'_\eta 
            & X''_\eta. \arrow [l, swap]
        \end{tikzcd}
        \end{center}
        Take $M'\leqslant \PsAut(\tilde{X}/S,\tilde{\Delta})$ to be the pseudo-automorphisms that descend to $\PsAut(X/S,\Delta)$, acting trivially on $V(X/S)$, and to $\Aut(X''_\eta)$ acting by $\Alb(X''_\eta)$-translations. For any $\phi\in M'$, because $\phi_\eta$ descends to $X''_\eta$, where it acts by $\Aut(X''_\eta)$-translation, $\phi_\eta$ acts trivially on $N^1(\tilde{X}_\eta)$.
        Therefore, $\phi_\eta$ acts  trivially on $N^1(X_\eta)$, giving (1). Moreover, fixing $D_0$ a divisor on $X$ such that $D_{0,\eta}\in\mathcal{A}(X_\eta)$, we get a group homomorphism $\sigma':\: M\to N^1(X/S)/V(X/S)$ defined by $\theta \mapsto [\theta_*D_0 - D_0]$ such that $\sigma'(M)$ is full rank inside of $\ker(r_V)$, as before, giving (2) and (3). 
\end{proof}

\begin{notation}
Let $(X/S,\Delta)$ be a klt CY pair of relative dimension two.
Assume that $\kappa(X_\eta,-K_{X_\eta})\leq 1$.
Then, we define the group $G(X/S,\Delta)$ as follows.
If $H^1(X_\eta,\mathcal{O}_{X_\eta})=0$, then we set $G(X/S,\Delta)=\{\id\}$.
Otherwise, we know that Proposition~\ref{prop:k=0}.(ii) holds and $G(X/S,\Delta)$ is the subgroup of 
$\sigma({\rm PsAut}(X/S,\Delta))\leqslant {\rm GL}(N^1(X/S),\zz)$ constructed in such proposition.
\end{notation} 

We proceed to the proof of the main result of the article.

\begin{theorem}
\label{thm:mainresult2}
     Let $(X/S,\Delta)$ be a klt CY pair of relative dimension two, then the geometric cone conjecture holds true for $(X/S,\Delta)$. That is, there are finitely many Mori chambers and Mori faces in $\mathcal{M}^e(X/S)$ up to $\PsAut(X/S,\Delta)$. Hence, there are finitely many marked SQMs $\alpha:\: X \dashrightarrow X_0$ over $S$ contractions $\phi:\: X_0\to T$ over $S$ up to $\PsAut(X/S,\Delta)$.
\end{theorem}

\begin{proof}
    Let $I(X/S)$ be the set of cones $\alpha^{-1}_*g_1^* \mathcal{A}(Y'/S)$ for those marked SQMs $\alpha:\: X\dashrightarrow X'$ of $X/S$ whose structure morphism factors as $X'\xrightarrow{g_1} Y'\xrightarrow{g_2} T \xrightarrow{h} S$ with $\dim T = \dim X - 1$ or $\dim X - 2$, and $h$ is not the identity. The cone $\alpha^{-1}_*g_1^* \mathcal{A}(Y'/S)$ is the interior of a chamber if $g_1$ is an isomorphism and the relative interior of a face if $g_1$ is a nontrivial contraction. In other words, the set $I(X/S)$ is the set of faces and chambers that themselves contain faces corresponding to SQM fiber space structures.\\

    \noindent\textit{Claim:} There exist finitely many $\PsAut(X/S,\Delta)$-orbits among the cones in $I(X/S)$.\\

    \begin{proof}[Proof of the Claim]
    For a given fiber space structure $g'\colon X'\rightarrow T$ of $X/S$ of relative dimension one, applying \cref{thm:reldim1}  to $g'\colon (X',\alpha_*\Delta) \rightarrow T$ tells us that there are only finitely many faces and chambers $\alpha^{-1}_*g_1^* \overline{\mathcal{A}}(Y'/S)$ of $\mathcal{M}^e(X/S)$ containing the face corresponding to $g'\colon X'\rightarrow T$ up to $\PsAut(X/S,\Delta)$. 
    Analogously, for a given fiber space structure
    $g'\colon X'\rightarrow T$ of $X/S$ of relative dimension two, induction on $V(X/S)$ shows that there are only finitely many faces and chambers $\alpha^{-1}_*g_1^* \overline{\mathcal{A}}(Y'/S)$ of $\mathcal{M}^e(X/S)$ containing the face corresponding to $g'\colon X'\rightarrow T$ up to $\PsAut(X/S,\Delta)$.
    Note that if $v(X/S)=0$, then there are no fiber space structures of $X/S$ of relative dimension two. Thus, there are finitely many faces and chambers up to $\PsAut(X/S,\Delta)$ containing a particular face corresponding to an SQM fiber space structure. By \cref{prop:FFSS}, there are finitely many chambers and faces corresponding to SQM fiber space structures $\PsAut(X/S,\Delta)$, so there are finitely many cones in $I(X/S)$ up to $\PsAut(X/S,\Delta)$ altogether.
    \end{proof}

    We now prove finiteness up to $\PsAut(X/S,\Delta)$ for chambers and faces not in the set $I(X/S)$. Define
    \[
    J(X/S,\Delta)=\tilde{\Pi}_1(X/S,\Delta) \setminus \bigcup I(X/S).
    \]
    We will prove in \cref{lem:proper} that the quotient 
    \[
    p:\: J(X/S,\Delta)\to W(X/S,\Delta)\subset N^1_1(X/S)/V(X/S)
    \]
    is proper, so the same holds for the induced map 
    \[
    J(X/S,\Delta)/G(X/S,\Delta)\to W(X/S,\Delta)/G(X/S,\Delta).
    \]
    In \cref{lem:k=1} and \cref{prop:k=0}, we proved that $W(X/S,\Delta)/G(X/S,\Delta)$ maps under $N_1^1(X/S)/V(X/S)\to N^1_1(X_\eta)$ either isomorphically onto $\Pi_1(X_\eta,\Delta_\eta)\cap \mathcal{B}(X_\eta)$ or as a real torus bundle over $\Pi_1(X_\eta,\Delta_\eta)\cap \mathcal{B}(X_\eta)$. In particular,
    the map $W(X/S,\Delta)/G(X/S,\Delta)\to \Pi_1(X_\eta,\Delta_\eta)\cap \mathcal{B}(X_\eta)$ is proper and hence the composition $$J(X/S,\Delta)/G(X/S,\Delta)\to W(X/S,\Delta)/G(X/S,\Delta)\to \Pi_1(X_\eta,\Delta_\eta)\cap \mathcal{B}(X_\eta)$$ is proper.   
    
    Moreover, the image of $J(X/S,\Delta)/G(X/S,\Delta)\to \Pi_1(X_\eta, \Delta_\eta)\cap \mathcal{B}(X_\eta)$ lies inside of a compact subset. Indeed, let $K_1\subseteq K_2\subseteq \cdots$ be a compact exhaustion of $\Pi_1(X_\eta, \Delta_\eta)\cap \mathcal{B}(X_\eta)$ such that every $K_i$ contains the faces of $\Pi_1(X_\eta, \Delta_\eta)\cap \mathcal{B}(X_\eta)$. In \cref{lem:closed}, we will see that $J(X/S,\Delta)$ is a closed subset of $N^1_1(X/S)$ contained in $\mathcal{B}_1(X/S)$. We have that $\partial (\bigcup I(X/S))\cap \partial J(X/S,\Delta)\subset \mathcal{B}(X/S)$ is closed in $N^1(X/S)$, so there are only finitely many faces in $\partial (\bigcup I(X/S))\cap \partial J(X/S,\Delta)$. Since each face in this intersection maps to a closed subset of $\Pi(X_\eta, \Delta_\eta)\cap \mathcal{B}(X_\eta)$ by properness, there exists an $N>0$ such that all of the faces of $\partial (\bigcup I(X/S))\cap \partial J(X/S,\Delta)$, (and therefore all of $J(X/S,\Delta)$) are in the preimage of $K_N$.
    
    Since $J(X/S,\Delta)/G(X/S,\Delta)\to \Pi_1(X_\eta, \Delta_\eta)\cap \mathcal{B}(X_\eta)$ is proper with image contained in a compact set, $J(X/S,\Delta)/G(X/S,\Delta)$ is compact. Therefore the Mori faces and chambers of $J(X/S)$ whose restriction to $X_\eta$ intersects $\Pi(X_\eta, \Delta_\eta)\cap \mathcal{B}(X_\eta)$ only fall into finitely many $G(X/S,\Delta)$-orbits, whence finitely many $\PsAut(X/S,\Delta)$-orbits. Up to $\Aut(X_\eta,\Delta_\eta)$, the restriction to $X_\eta$ of every Mori face in $\mathcal{B}(X/S)$ and every Mori chamber intersects $\Pi(X_\eta, \Delta_\eta)\cap \mathcal{B}(X_\eta)$/.
    However, this may not be the case up to $\PsAut(X/S,\Delta)$.

    For any klt CY pair, there is a finite index subgroup of $\Aut(X_\eta,\Delta_\eta)$ of automorphisms that extend to $\PsAut(X/S,\Delta)$. Indeed, any automorphism $\sigma\in \Aut(X_\eta,\Delta_\eta)$ extends uniquely to an birational automorphism of $\sigma'\in \Bir(X/S)$. We argue that the biratonal automorphism $\sigma'\in \Bir(X/S)$ is an element of $\Bir(X/S,\Delta)$.
    If $(X/S,\Delta')$ is the log pull-back of $(X/S,\Delta)$ with respect to $\sigma'$, then $K_X+\Delta'$ and $K_X+\Delta$ differ by the pull-back of a divisor $D$ on $S$. Indeed, both klt sub-pairs $(X/S,\Delta)$ and $(X/S,\Delta')$ agree over an open of $S$ and are log Calabi--Yau over $S$. On the other hand, the pair obtained on $S$ by the canonical bundle formula applied to $(X/S,\Delta)$ and $(X/S,\Delta')$ are equal, so $D=0$. We conclude that $\sigma'\in {\rm Bir}(X/S,\Delta)$ and so it induces a permutation of the finite set of non-terminal exceptional valuations of $(X/S,\Delta)$. 
    The element $\sigma'$ is in ${\rm PsAut}(X/S,\Delta)$ whenever such action is the identity.
    Therefore, there is a finite index subgroup of 
    $\Aut(X_\eta,\Delta_\eta)$ that lifts
    to $\PsAut(X/S,\Delta)$.
    
    Writing $\PsAut(X/S,\Delta)\leqslant \Aut(X_\eta,\Delta_\eta)$, choose $g_1,\hdots, g_s$ right coset representatives of $\PsAut(X/S,\Delta)$ in $\Aut(X_\eta,\Delta_\eta)$. For any $D\in \mathcal{M}^e(X/S)$, there exists a $g\in \Aut(X_\eta,\Delta_\eta)$ such that $g^*D_\eta\in \Pi(X_\eta,\Delta_\eta)$. We have that $g=h \circ g_i$ for some $i$ and for some $h\in \PsAut(X/S,\Delta)$. By construction, $h^* D_\eta\in (g_i^{-1})^* \Pi(X_\eta,\Delta_\eta)$. Therefore, up to $\PsAut(X/S,\Delta)$, any Mori chamber or Mori face in $\mathcal{M}^e(X/S)$ lies in $r^{-1}((g_i^{-1})^*\Pi(X_\eta,\Delta_\eta))$ for some $i=1,\hdots, s$. Henceforth, applying the above argument with $\Pi(X_\eta,\Delta_\eta)$ replaced with $(g_i^{-1})^* \Pi(X_\eta,\Delta_\eta)$ for each $i=1,\hdots,s$, we obtain that all of the Mori faces and chambers of $\mathcal{M}^e(X/S)$ fall into finitely many $\PsAut(X/S,\Delta)$-orbits. It follows that there are only finitely many $\PsAut(X/S,\Delta)$-orbits of Mori chambers. 
\end{proof}

In~\cref{lem:proper}, we will show that  $p:\: J(X/S,\Delta)\to W(X/S,\Delta)$ is proper.
To do so, we prove the following lemma in order
for bounded sequences in $J(X/S,\Delta)$ to admit convergent subsequences.

\begin{lemma}
\label{lem:closed}
    $J(X/S,\Delta)$ is a closed subset of $N^1_1(X/S)$ contained in $\mathcal{B}_1(X/S)$.
\end{lemma}

\begin{proof}
    Since $J(X/S,\Delta)$ is a closed subset of $\mathcal{B}(X/S)$, the only way $J(X/S,\Delta)$ may fail to be closed in $N^1_1(X/S)$ is if there is a class $z\in \overline{J(X/S,\Delta)}$ that is not big over $S$. Bigness over $S$ is detected on the generic fiber, so we show that there is no element $z\in J(X/S,\Delta)$ such that $w:=r(z)\in N^1(X_\eta)$ has Iitaka dimension at most one, or equivalently, $r(z)^2=0$. We proceed by contradiction, assuming the existence of such a class $z$.\\ 

    \noindent\textit{Claim:} If such a class $z$ exists, then $G(X/S,\Delta) = \{\id\}$, and hence $\rho = \rho(X/S) = \rho(X_\eta) + v(X/S)$.\\
    
    \begin{proof}[Proof of the Claim]By the construction of $G(X/S,\Delta)$, it suffices to show this for $X_{\bar{\eta}}$ is abelian, hyperelliptic, or $(E\times \mathbb{P}^1)/G$. Suppose $X_{\bar{\eta}}$ is abelian and hence the $\mathcal{A}^e(X_\eta)$ coincides with the positive cone 
    \[
    \{z\in N^1(X_\eta)\mid z^2> 0
    \text{ and } z\cdot H>0\}
    \]
    for some hyperplane class $H$. The boundary of the positive cone, those square-zero classes, forms a quadric cone in $N^1(X/S)$. The existence of a rational point $w$ implies that that this quadric cone is rational, and therefore we may find a finite set of classes $\{w_i\}_i\in\mathcal{A}^e(X_\eta)$ generating $N^1(X_\eta)$ such that $w_i^2=0$. Taking closures we extend these to classes $w_i'\in N^1(X/S)$ which yield fiber space structures $X\xrightarrow{g_i} T_i \xrightarrow{h_i} S$. We have that $w_i = (g_i^* v_i)_\eta$ for $h_i$-ample classes $v_i\in N^1(T_i/S)$. Since each element of $G(X/S,\Delta)$ acts trivially on $N^1(X_\eta)$, it fixes each $w_i$ and therefore induces an automorphism on each $T_i$. We may find a finite index subgroup of $G(X/S,\Delta)$ fixing each $v_i$. Because $G(X/S,\Delta)$ acts as a group of translations on the fibers of $r:\: N^1(X/S) \to N^1(X_\eta)$, this finite index subgroup fixing each $v_i$ acts trivially on $N^1(X/S)$. Therefore, $G(X/S,\Delta)$ must be acting trivially as well. Hence, $G(X/S,\Delta) = \{\id\}$. 
   
   For $(E\times \mathbb{P}^1)/G$ an isotrivial elliptic fibration over $\mathbb{P}^1$, the group $G$ acts on $N^1(E\times \mathbb{P}^1)$ sending non-big classes to non-big classes. Hence, it must fix the classes defining the two canonical projections. Otherwise, the quotient would have Picard rank $1$, contradicting the existence of the isotrivial elliptic fibration. Therefore, $(E\times \mathbb{P}^1)/G$ has Picard rank $2$ with two distinct fiber space structures. The same is true if $X_{\bar{\eta}}$ hyperelliptic, so in these two cases, we may continue the argument similarly to the case where $X_{\bar{\eta}}$ is abelian. Hence, the claim is proven.
   \end{proof}

   By \cref{lem:nicefactorization}, after replacing $(X/S,\Delta)$ with an SQM, we may assume that $(X/S,\Delta)$ factors as $X\xrightarrow{g} T \xrightarrow{h} S$ where $g$ is a good fibration and $h$ is birational. Let $\rho = \rho(X/S)$, $\rho_\eta = \rho(X_\eta)$, and $k=\rho(T/S)$. 
   
   Extend $w$ to a basis of $N^1(X_\eta)$ and lift each basis element under $r$ to obtain $x_1,\hdots,x_{\rho_\eta}\in N^1(X/S)$ with $r(x_{\rho_\eta})=w$. Let $x_{\rho_\eta + k},\hdots, x_{\rho}\in V(X/S)$ be the classes of the $g$-exceptional prime divisors. These are the degenerate divisors of $X/S$ of insufficient fiber type. Indeed, $T$ is obtained by blowing up $S$ so that the degenerate divisors mapping to codimension $\geq 2$ on $S$ map to divisors in $T$, and then taking a small $\mathbb{Q}$-factorialization. Since $k=v(T/S) + 1$, the variety $T$ has $k-1$ prime $h$-exceptional divisors we denote $y_2,\hdots, y_k$. Complete $x_{\rho_\eta + k},\hdots, x_{\rho}\in V(X/S)$ to a basis $x_{\rho_\eta},\hdots, x_{\rho}$ of $V(X/S)$ by letting $g^* y_j = x_{\rho_\eta+j-1}$ for $1<j\leq k$. Altogether, we obtain a basis $x_1,\hdots, x_\rho$ of $N^1(X/S)$ such that:
   \begin{enumerate}
        \item $r(x_1)\hdots,r(x_{\rho_\eta})$ are a basis of $N^1(X_\eta)$;
        \item $w=r(x_{\rho_\eta})$;
        \item $x_{\rho_\eta+1},\hdots, x_{\rho}$ are a basis of $V(X/S)$; and
        \item $x_{\rho_\eta+k},\hdots, x_{\rho}$ are $g$-exceptional (of insufficient fiber type).
   \end{enumerate}
   
   Lifting $x_{\rho_\eta}$ under $g^*$ to an element $y_1\in N^1(T/S)$, we also obtain a basis $y_1,\hdots, y_k$ of $N^1(T/S)$ such that:
    \begin{enumerate}
        \item[(i)] $g^* y_1 = x_{\rho_\eta}$; and
        \item[(ii)] $g^* y_j = x_{\rho_\eta+j-1}$, with $y_j$ $h$-exceptional, for $1<j\leq k$.
    \end{enumerate}
    
    By assumption, $z\in \overline{J(X/S,\Delta)}$, so we may take $\{z_n\}$ a sequence in $J(X/S,\Delta)$ converging to $z$. Decompose each $z_n$ in terms of our basis for $N^1(X/S)$:
    \begin{align*}
        z_n &= \sum_{i=1}^{\rho} a_{n,i}x_i.
    \end{align*}

    We know that $(r(z_n))_{n\in \mathbb{N}} = (r(\sum_{i=1}^{\rho_\eta} a_{n,i}x_i))_{n\in \mathbb{N}}$ must converge to $r(z)=r(x_{\rho_\eta})=w$, so $a_{n,i}\to 0$ as $n\to \infty$ for $1\leq i<\rho_\eta$ (and $a_{n,\rho_\eta}\to 1$). If we define $c_n = \|(a_{n,1}, \hdots, a_{n,\rho_\eta-1})\|^{-1}$, then $\sum_{i=1}^{\rho_\eta-1} c_na_{n,i}x_i$ lies on the unit sphere in $\mathbb{R}\cdot x_1 \oplus \cdots \oplus \mathbb{R}\cdot x_{\rho_\eta -1}$, so by compactness contains a convergent subsequence. We pass to this subsequence. Hence, we have a sequence $(c_n)_{n\in\mathbb{N}}$ of positive real numbers approaching infinity such that $z'_n = \sum_{i=1}^{\rho_\eta-1} c_na_{n,i}x_i$ converges to a nonzero class $z'\in N^1(X/S)$.  

    For $\rho_\eta + k \leq j\leq \rho$, let $C_j$ denote a curve contained in a fiber of the prime exceptional divisor $x_j$ over its image in $T$. We have that 
    \begin{align*}
    \left(\sum_{i=1}^{\rho} c_n a_{n,i}x_i \right)\cdot C_j &\geq 0
    \end{align*}
    since $c_nz_n\in \mathcal{M}^e(X/S)$. On the other hand, by push-pull $x_i\cdot C_j = 0$ for $\rho_\eta\leq i\leq \rho$ except for $i=j$.
    Indeed, this follows as $x_{\rho_\eta}$ is pulled back from $y_1\in N^1(T/S)$, and $x_j$ is of insufficient fiber type. So 
    \begin{align*}
    \left(\sum_{i=1}^{\rho} c_n a_{n,i}x_i\right)\cdot C_j &= \left(\sum_{i=1}^{\rho_\eta-1} c_n a_{n,i}x_i\right)\cdot C_j + c_na_{n,j} x_j \cdot C_j.
    \end{align*}

    Since $x_j \cdot C_j<0$, and $(\sum_{i=1}^{\rho_\eta-1} c_n a_{n,i}x_i)\cdot C_j$ approaches $z'\cdot C_j$, the value $c_n a_{n,j}$ must be bounded as well. As $c_n a_{n,j}$ are bounded as $n\to \infty$ for $\rho_\eta + k \leq j\leq \rho$, by passing to another subsequence, 
    \begin{align*}
    z''_n &:= \sum_{i=\rho_\eta+k}^{\rho} c_n a_{n,i}x_i
    \end{align*}
    converges to some $z''\in N^1(X/S)$. 

    By \cref{cor:finite}, the cone $\mathcal{A}^e(T/S)$ is a finite rational polyhedral cone. Let $r_i$ for $1\leq i\leq e$ be the generators of this cone. 
    We may write
    \begin{align*}
    \sum_{j=1}^{k} c_n a_{\rho_\eta +j-1} y_j &= \sum_{i=1}^e b_{n,i} r_i 
    \end{align*}

    for some non-uniquely determined $b_{n,i}$. We choose $b_{n,i}$ so that $\{b_{n,i}\}_{n,i}$ is bounded from below. Rearrange indices so that $b_{n,i}\to \infty$ for all $1\leq i\leq e'$ and some $e'\leq e$, and 
    $y'_n := \sum_{i=e'+1}^e b_{n,i} r_i$ converges to some $y'$. Because $c_n$ goes to infinity, we have that $e'\geq 1$. Observe that $\sum_{i=1}^{e'} r_i$ is $h$-big; otherwise, by convexity, $C\cdot r_i=0$ for $C$ a general fiber of $h$ and $1\leq i \leq e'$, and then $C \cdot \sum_{i=1}^e c_n b_{n,i} r_i$ converges, contradicting $c_n\to \infty$.  Since $\sum_{i=1}^{e'} r_i$ is $h$-big and $h$-nef, it is $h$-semiample, defining a birational morphism $h_1$ over $S$. We obtain a factorization $T \xrightarrow{h_1} T' \xrightarrow{h_2} S$ and an $h_2$-ample class $d\in N^1(T'/S)$ such that $h_1^*d = \sum_{i=1}^{e'} r_i$. We have that 
    \[
    c_n z_n \equiv_{T'} z'_n + z''_n + g^*y'_n \in \mathcal{M}^e(X/T').
    \]
    Taking limits, we have that $z' + z'' + z''' \in \mathcal{M}^e(X/T')$ as well, defining $z''' := g^* y'$. Since $z_n\in J(X/S,\Delta)$, $z' + z'' + z''' \in \mathcal{M}^e(X/T') \cap \mathcal{B}(X/T')$. We may find a marked SQM $\alpha:\: X\dashrightarrow X'$ such that $\alpha_*(z'+z''+z''')\in\mathcal{A}^e(X'/T')$. The class $\alpha_*(z'+z''+z''')$ is thus semiample on $X'$ over $T'$ and we may form a factorization $X'\xrightarrow{g_1} Y'\xrightarrow{g_2} T'$ such that $g_1$ is birational and $z'+z''+z''' = g_1^* \ell$ for some $g_2$-ample class $\ell\in N^1(Y'/T')$. Taking $N$ large enough so that $\ell + Ng_2^*d$ is $h_2 \circ g_2$-ample, we have that $g_1^*(\ell + Ng_2^*d)$ yields a lift of $z'+z''+z'''$ in $\mathcal{A}^e(X'/S)$, and $\alpha_*^{-1} g_1^*\mathcal{A}(Y'/S)\in I(X/S)$.
    Passing to a subsequence and applying \cref{prop:locallyfinite} to $X'\xrightarrow{g_2\circ g_1} T'$
    we may assume that $[z_n] \in \mathcal{A}^e(X'/T')$ for all $n$.
    Since $b_{n,i} \to \infty$ for $1 \leq i \leq e'$, we have
    $z_n \in\bigcup I(X/S)$, a contradiction.
\end{proof}

For the following crucial lemma we adapt the technique in step 5 of the proof of \cite[Theorem 3.4]{FHS21}.

\begin{lemma}
\label{lem:proper}
    The map $p:\:J(X/S,\Delta)\to W(X/S, \Delta)$ is proper. 
\end{lemma}

\begin{proof}
    For the sake of contradiction, suppose that there exists a sequence $(z_n)_{n\in\mathbb{N}}$ of classes in $J(X/S,\Delta)$ such that $(p(z_n))_{n\in\mathbb{N}}$ converges, but $(z_n)_{n\in\mathbb{N}}$ does not admit any convergent subsequences. Both of these properties are preserved under passing to subsequences. Every time we prove the existence of a certain subsequence, we understand that it is refining all previous subsequences, and we pass to it for all objects referring to the subsequence of indices as $n\in\mathbb{N}$.

    Because it suffices to prove this result on a SQM of $f:\:(X,\Delta)\to S$, by \cref{lem:nicefactorization}, we assume without loss of generality that $f=h\circ g$ where $g:\: X\to S'$ is a good fibration and $h:\:S'\to S$ is birational. Write
    \begin{align*}
    k&:= \rho(X/S')\\
    v&:=v(X/S)\\
    \rho&:=\rho(X/S).
    \end{align*}   
    By \cite[Lemma 2.5]{FHS21}, \cite[Lemma 2.18]{FHS21}, and \cite[Lemma 3.1.3]{FHS21}, the subspace $g^*N^1(S'/S)\subset V(X/S)$ is intrinsically defined and independent of the chosen factorization through a good fibration. Let $\Delta'$ be a divisor on $S'$ so that $(S'/S,\Delta')$ is the klt CY pair obtained by the canonical bundle formula for $(X,\Delta)$. In a similar manner to \cref{lem:closed}, we 
    turn to define a basis of $N^1(X/S)$ 
    that will allow us to understand the map $p:\: J(X/S,\Delta)\to W(X/S,\Delta)$. Begin with a basis $y_1,\hdots, y_k$ of $g^*N^1(S'/S)$, complete to a basis $y_1,\hdots, y_v$ of $V(X/S)$, and further complete to a basis $y_1,\hdots, y_\rho$ of $N^1(X/S)$. Write $z_n = \sum_{i=1}^\rho a^i_n y_i$ and define the summand sequences 
    \begin{align*}
    w_n&:= \sum_{i=1}^k a_{n,i} y_i\\
    x_n&:= \sum_{i=k+1}^\rho a_{n,i} y_i.\\
    \end{align*}

    We have that $z_n = x_n + w_n$ for all $n\in \mathbb{N}$. The classes $z_n$ and $x_n$ are big over $S$. Moreover, the subspace span$(y_{v+1},\hdots, y_\rho)$ maps isomorphically onto $N^1(X/S) / V(X/S)$.\\

    \noindent\textit{Step 1:} We show that 
    $(x_n)_{n\in \mathbb{N}}$ contains a convergent subsequence.\\ 

    By construction, we have $[z_n]=[x_n]\in N^1(X/S')$, so it suffices to show that $[z_n]=\sum_{i=k+1}^\rho a_{n,i}[y_i]$ contains a convergent subsequence. We have that $[z_n]=\sum_{i=k+1}^v a_{n,i}[y_i] + \sum_{i=v+1}^\rho a_{n,i}[y_i]$. The classes $y_{k+1},\hdots,y_v$ are the classes of prime $g$-exceptional divisors. Since $g$ is a good fibration, these divisors are of insufficient fiber type. Moreover, $a_{n,v+1},\hdots, a_{n,\rho}$ converge by our assumption on the convergence of $p(z_n)=\sum_{i=v+1}^\rho a_{n,i}p(y_i)$. We must then show that each $a_{n,k+1},\hdots, a_{n,v}$ is bounded. Suppose for contraction that at least one of these is not bounded. Then, we have that $[z_n] = [u_n] + [b_n]$ where $u_n$ are the classes of insufficient fiber type divisors with diverging coefficients, and $[b_n]$ converges. Since $[z_n]$ is movable, and, by this decomposition, converges to $V(X/S')$, $[z_n]$ must be approaching $0\in N^1(X/S')$ by \cref{prop:movert} in the Addendum. However, this is impossible, because $[z_n]$, written in our basis of $N^1(X/S')$, contains summands with diverging coefficients, a contradiction. \\

    \noindent\textit{Step 2:} We show that there exists a birational contraction $S'\dashrightarrow S_N$ over $S$ that is the outcome of a $w'_n$-MMP for a suitable subsequence of $(w'_n)_{n\in \mathbb{N}}$.
    Moreover, we define some divisor classes in the models that appear in these MMPs.\\

    For any $n$, we may run a $w'_n$-MMP over $S$ and it will terminate with a minimal model over $S$. By~\cref{lem:crepant}, there are only finitely many birational models of $S$ that may appear. By the negativity lemma, no model may appear more than once in the $w'_n$-MMP. Hence, we may find a subsequence such that there exists a sequence of pseudo-isomorphisms and divisorial contractions $\{\psi_i:\: S_i\dashrightarrow S_{i+1}\}_{i=0}^N$ with $S_0:=S'$ over $S$ that gives a run of the $w'_n$-MMP for all $n$. In particular, we have that $(\psi_{N-1}\circ \cdots\circ \psi_0)_* w'_n$ is nef over $S$ for all $n$. Moreover, for each $i$ by \cref{lem:existssqm} there exists a marked SQM $(f_i:\: X_i \to S,\alpha_i)$ of $X/S$ together with a factorization 
        \begin{center}
        \begin{tikzcd}
            X_i \arrow[r, "l_i"] \arrow[rr, bend left=30, "f_i"] 
            & S_i \arrow[r] & S
        \end{tikzcd}
        \end{center}
        such that the square 
        \begin{center}
                \begin{tikzcd}
                 X_i \arrow[r, "\alpha_{i+1}^{-1}\circ\alpha_i", dashed] \arrow[d, "l_i", swap] & X_{i+1} \arrow[d, "l_{i+1}"] \\
                 S_i \arrow [r, "\psi_i", swap, dashed] & S_{i+1}
                \end{tikzcd}
        \end{center}
    commutes. Write $\phi_i:= \alpha_{i+1}^{-1}\circ\alpha_i$. Passing to another subsequence, we may assume that each $w'_n$ has the same ample model $\psi_N:\: S_N \to S_{N+1}$. Now, we define analogous sequences of numerical classes on each $X_i$. Let  $w_{n,0} := w_n$, $w'_{n,0} := w'_n$, $z_{n,0} := z_n$, and $x_{n,0} := x_n$. Define
    \begin{align*}
        z_{n, i+1} &:= (\phi_i)_* z_{n, i} = (\alpha_{i})_* z_n \in N^1(X_{i+1}/S),\\
        w'_{n,i+1} &:= (\psi_i)_* w'_{n,i} \in N^1(S_{i+1}/S), \\ 
        w_{n,i+1} &:= l_{i+1}^* w'_{n,i+1} \in N^1(X_{i+1}/S), \text{ and }\\
        x_{n, i+1} &:= z_{n,i+1}-w_{n,i+1} \in N^1(X_{i+1}/S).
    \end{align*}
    It follows immediately that $$z_{n,i}=x_{n,i}+w_{n, i}\in N^1(X_i/S)$$ and $$[z_{n,i}] = [x_{n,i}] \in N^1(X_i/S_i)$$ for all $i=0, 1, \dots, N$. Define $w'_{n,N+1}$ to be the ample class in $N^1(S_{N+1}/S)$ such that $\psi_N^*w'_{n,N+1} = w'_{n,N}$. Since  $w'_{n,N}$ is pulled back from $S_{N+1}$, we have that 
    \[
    [z_{n,N}] = [x_{n,N}] \in N^1(X_N/S_{N+1}).
    \]

    \noindent\textit{Step 3:} We show that for each $i=0,1,\dots,N$, we may find a converging subsequence of $(x_{n,i})_{n\in \mathbb{N}}$ whose indices are chosen independently of $i$.\\

    We prove this by induction on $i$. The base case being Step 1. If $\psi_{i}$ is a flip, then there is nothing to prove as $x_{n,i+1} = (\phi_{i})_* x_{n,i}$ and $(\phi_{i})_*$ induces an isomorphism between $N^1(X_{i}/S_{i})$ and $N^1(X_{i+1}/S_{i+1})$. If $\psi_{i}$ is a divisorial contraction, then write $E_{i}$ for its exceptional divisor. We have that 
    \[
    w'_{n,i} - \psi^*_{i+1} w'_{n,i+1} = c_{n,i} [E_{i}]\in N^1(S_{i}/S)
    \]
    for some $c_{n,i}>0$. Since $E_{i}$ is exceptional and $c_{n,i}>0$, $c_{n,i} [E_{i}]\notin \overline{\mathcal{M}}(S_{i}/S_{i+1})$ (e.g., \cite[Lemma 2.11]{FHS21}). If we set $F_{i} := l_i^* E_{i}$, then $[z_{n,i}] = [\phi_i^* x_{n,i+1}] = [x_{n,i} + c_{n,i} F_{i}]\in N^1(X_i/S_{i+1})$. Since $z_{n,i}\in \overline{\mathcal{M}}(X_i/S)$, and because movability over $S_{i+1}$ is a weaker condition than movability over $S$ (e.g., \cite[Lemma 2.2]{FHS21}), we have that $[x_{n,i} + c_{n,i} F_{i}]\in \overline{\mathcal{M}}(X_i/S_{i+1})$. On the other hand, $c_{n,i} [F_{i}]\notin \overline{\mathcal{M}}(X_i/S_{i+1})$, because $c_{n,i} [E_{i}]\notin \overline{\mathcal{M}}(S_{i}/S_{i+1})$. The coefficients $c_{n,i}$ are thus bounded above, so they must contain a convergent subsequence. By construction, we have that $w_{n,i+1} = (\phi_i)_* w_{n,i} - c_{n,i} (\phi_i)_* F_i$, so $x_{n,i+1} = (\phi_i)_* x_{n,i} + c_{n,i} (\phi_i)_* F_i$. Therefore, taking a subsequence where $c_{n,i}$ converges yields a subsequence where $x_{n,i+1}$ does too, completing our inductive step.

    We observe that this step implies that $S_{N+1}\to S$ is not an isomorphism. Indeed, if there is an isomorphism $S_{N+1}\xrightarrow{\sim} S$, then $w_{n,N}$ is the pull-back of a divisor from $S$. In this case, $z_{n,N} = x_{n,N} + w_{n,N} = x_{n,N} \in N^1(X_N/S)$. Passing to the subsequence in the claim, $x_{n,N}$ converges, whereas $z_{n,N}$ may not contain any convergent subsequences, a contradiction.\\

    \noindent\textit{Step 4:} We show that there exists a marked SQM $(f_{N+1}:\: X_{N+1}\to S, \alpha_{N+1})$ of $f$ with a factorization
    \begin{center}
        \begin{tikzcd}
            X_{N+1} \arrow[r, "l_{N+1}"] \arrow[rr, bend left=30, "f_{N+1}"] 
            & S_{N+1} \arrow[r] & S
        \end{tikzcd}
        \end{center}
    such that for a subsequence of indices $n\in 
    \mathbb{N}$, the class $[x_{n,N+1}]$ is nef over $S_{N+1}$, where $x_{n,N+1}:=(\phi_N)_* x_{n,N}$ and $\phi_N := \alpha_{N+1}^{-1}\circ \alpha_N$. \\

        Since $z_{n,N}\in (\alpha_N)_* J(X/S)$ and $[z_{n,N}]=[x_{n,N}]\in N^1(X_N/S_{N+1})$ for each $n$, the class $x_{n,N}$ is big over $S$ and therefore it is big over $S_{N+1}$ as well. Applying \cref{prop:locallyfinite} to the morphism $X_N\to S_{N+1}$ finishes this step.
        We define $z_{n,N+1}=(\phi_N)_* z_{n,N}$ and $w_{n,N+1}=l_{N+1}^* w'_{n,N+1}$.\\

        \noindent\textit{Step 5:} We show that there exists a positive real number $\epsilon_{N+1}$ such that for any curve $C\subset X_{N+1}$ contained in a fiber of $f_{N+1}$, we have the inequality 
        \[
        x_{n,N+1}\cdot C > -\dfrac{2\dim X_{N+1}}{\epsilon_{N+1}}.
        \]

        Applying the argument from Step 4 as well as \cite[Lemma 2.4]{FHS21} to $f_{N+1}$, we may find for a subsequence of indices $n\in \mathbb{N}$ effective $\mathbb{R}$-divisors $D_{n,N+1}$, big over $S$, such that $x_{n,N+1} = [D_{n,N+1}]\in N^1(X_{N+1}/S)$, and an $0<\epsilon_{N+1}<1$ such that each $(X_{N+1}, \epsilon_{N+1} D_{n,N+1})$ is klt. Therefore, the claim follows immediately from the the cone theorem.
        We fix an integer $T_{N+1}\geq 2\dim X_{N+1}\epsilon_{N+1}^{-1}$ and a Cartier divisor $H_{N+1}$ on $S_{N+1}$ ample over $S$.\\

        \noindent\textit{Step 6:} We show that for infinitely many $n\in \mathbb{N}$, we have $z_n\in \bigcup I(X/S)$.\\

        We will proceed in two cases, depending on the nefness of the divisors $w'_{n,N+1}-T_{N+1}H_{N+1}$.\\

        \textit{Case 1:} Assume that we may pass to a subsequence of $n\in \mathbb{N}$ to arrange that $w'_{n,N+1} - T_{N+1}H_{N+1}$ is nef over $S$.\\

        Then, for any irreducible curve $C'\subset S_{N+1}$ in a fiber of $S_{N+1}\to S$, we have that $w'_{n,N+1}\cdot C' \geq T_{N+1}$. For any irreducible curve $C\subset X_{N+1}$ contained in a fiber of $f_{N+1}:\: X_{N+1}\to S$, its image $l_{N+1}(C)$ is either a point or a curve contained in a fiber of $S_{N+1}\to S$, so $w_{n,N+1}\cdot C$ is either $0$ or at least $T_{N+1}$. If $C$ is contracted by $l_{N+1}$, then $z_{n,N+1}\cdot C = x_{n,N+1}\cdot C \geq 0$, because we chose $X_{N+1}$ in Step 4 so that $x_{n,N+1}$ is nef over $S_{N+1}$. If $l_{N+1}(C)$ is a curve, then by Step 5, we at least have $x_{n, N+1}\cdot C> -T_{N+1}$. Hence, $z_{n,N+1}\cdot C = (x_{n,N+1} + w_{n,N+1})\cdot C > 0$, and so the class $z_{n,N+1}$ is nef over $S$, only pairing trivially with curves contained in the fibers over $S_{N+1}$. Thus, the class $z_{n,N+1}$ defines a factorization $X_{N+1}\to X'_{N+1}\to S_{N+1}\to S$ and is the pullback of a divisor on $X'_{N+1}$ ample over $S$. This is a contradiction to our initial assumption of $z_n\notin \bigcup I(X/S)$.\\
        
        \textit{Case 2:} Suppose now that we may not pass to a subsequence to arrange that $w'_{n,N+1} - T_{N+1}H_{N+1}$ is nef over $S$.\\
        
        Setting 
        \[
        \lambda_n := \inf \left\{\lambda\in \mathbb{R} \mid  w_{n,N+1}' + (\lambda - 1)T_{N+1} H_{N+1} \text{ is nef over } S \right\}\in [0,1],
        \]
        fix an extremal contraction $\psi_{n,N+1}:\: S_{N+1} \to S_{N+2}$ over $S$ contracting a ray $R_n$ of $\overline{NE}(S_{N+1}/S)$ whose intersection with $w'_{n,N+1} - (\lambda_n-1) T_{N+1}H_{N+1}$ pairs to zero. We apply the canonical bundle formula to obtain an effective divisor $\Delta_{S_{N+1}}$ on $S_{N+1}$ such that $(S_{N+1}/S,\Delta_{S_{N+1}})$ is klt CY.
        Using~\cref{lem:crepant}, we pass to a subsequence of indices $n\in \mathbb{N}$ such that:
        \begin{enumerate}
            \item $\psi_{n,{N+1}}$ are all the same, and
            \item $\lambda_n$ converges.
        \end{enumerate}

        Set $\psi_{N+2}:=\psi_{n,N+2}$ for any choice of $n\in \mathbb{N}$. Setting 
        \[
        x_{n,N+1}':= x_{n,N+1} + (1-\lambda_n)l^*_{N+1} T_{N+1} H_{N+1}\in N^1(X_{N+1}/S),
        \]
        we have that $(x_{n,N+1}')_{n\in \mathbb{N}}$ converges. Since all $\psi_{n,N+1}$ are the same, we have that $$w_{n,N+1}-(1-\lambda_n)l^*_{N+1} T_{N+1} H_{N+1}\equiv_{S_{N+2}}0$$ for all $n\in \mathbb{N}$. Therefore, we have that $[z_{n,N+1}]=[x_{n,N+1}']\in N^1(X_{N+1}/S_{N+2})$, implying that $S_{N+2}\to S$ is not an isomorphism, because if this were the case then $(z_n)_{n\in \mathbb{N}}$ would converge, contradicting our assumption. 

        Now, we apply the previous argument inductively. The proof of Step 4 implies that there exists a marked SQM $(f_{N+2}, X_{N+2}\to S,\alpha_{N+2})$ and a commutative diagram 
        \begin{center}
        \begin{large}
        \begin{tikzcd}
            X_{N+1} \arrow[d, "l_{N+1}", swap] \arrow[rr,"\phi_{N+1}",dashed]
            && X_{N+2} \arrow[d, "l_{N+2}"] \\ 
            S_{N+1} \arrow[rr, "\psi_{N+1}"] \arrow[dr]
            && S_{N+2} \arrow[dl]\\
            & S &
        \end{tikzcd}
        \end{large}
        \end{center}

        such that $x'_{n,N+2}:= (\phi_{N+1})_* x'_{n,N+1}$ is nef over $S_{N+2}$ for $n\in \mathbb{N}$ after passing to a subsequence. As in Step 5, there exists $0<\epsilon_{N+2}<1$ such that $x_{n,N+2}\cdot C> 2\dim X_{N+2}\epsilon_{N+2}^{-1}$ for any curve $C$ contained in the fibers of $f_{N+2}$. Fixing $T_{N+2}\geq 2 \dim X_{N+2}\epsilon_{N+2}^{-1}$ and $H_{N+2}$ a Cartier divisor on $S_{N+2}$ ample over $S$, we define $w'_{n,N+2}:=(\psi_{N+1})_* (w'_{n,N+1} + (\lambda_n -1)T_{N+1}H_{N+1})$ for each $n\in \mathbb{N}$. If there exists a subsequence such that 
        $w'_{n,N+2}-T_{N+2}H_{N+2}$ is nef over $S$, then we reach a contradiction as in the fist case of this step. Otherwise, we proceed inductively.
        
        The procedure yields a decreasing sequence of birational contractions $\psi_{N+i}$ fitting into a commutative diagram 

        \begin{center}
        \begin{large}
        \begin{tikzcd}
            S_{N+1} \arrow[r,"\psi_{N+1}"] \arrow[rrd]
            & S_{N+2} \arrow[r,"\psi_{N+2}"] \arrow[rd]
            & \cdots \arrow[r,"\psi_{N+l-2}"]
            & S_{N+l-1} \arrow[r,"\psi_{N+l-1}"] \arrow[ld]
            & S_{N+l} \arrow[lld]\\
            && S &&
        \end{tikzcd}
        \end{large}
        \end{center}

    such that $S_{N+i}\to S$ is not an isomorphism and at most one contraction $\psi_{N+i}$ is fiber type. Because $\rho(S_{N+i}/S)$ decreases at each step, the diagram must be finite and terminate at some $S_{N+l}\to S$. As the procedure must terminate, we know that there exists a subsequence of indices $n\in \mathbb{N}$ such that $w'_{n,N+l}-T_{N+l}H_{N+l}$ is nef over $S$.
    This contradicts the first case of this step.
    \end{proof}

    We finish this section by proving the main two statements of this paper.

    \begin{proof}[Proof of Theorem~\ref{thm:mainresult}]
    This follows immediately from Theorem~\ref{thm:mainresult2}.
    \end{proof}

    \begin{proof}[Proof of Corollary~\ref{cor:minmodels}]
    The proof is verbatim from~\cite[Corollary 1.7]{Li23} by replacing~\cite[Corollary 6.3]{Li23} with Theorem~\ref{thm:mainresult}.
    \end{proof}

\section{Addendum: A property of good fibrations}

In this section, we discuss a nice property of good fibrations.
In the proof of \cref{thm:mainresult}, one might hope that the restriction of the quotient map $\pi:\: N^1(X/S) \to N^1(X/S)/V(X/S)$ to $\overline{\mathcal{M}}(X/S)$ is proper, as in \cite[Lemma 4.2(3)]{Kaw97}. This would make \cref{lem:proper} superfluous and greatly shorten the argument. However, when $\dim S > 1$, or more generally when $f:\: (X,\Delta)\to S$ is not a good fibration, this is not the case. 

For instance, this is not the case if a prime divisor $D$ in $X$ is mapped to a codimension at least $2$ subset $f(D)$ in $S$ that is contained in a movable divisor $M$ on $S$. Suppose that $D$ is the only divisor for which this is the case. We have that $f^* M$ is reducible containing $D$ as an irreducible component. Taking $x:= f^* M - D$, we have that $x$ is nonzero in $N^1(X/S)$ and contained in the intersection $V(X/S)\cap \overline{\mathcal{M}}(X/S)$. It follows that all of $\mathbb{R}^{\geq 0}x$ is contained in this intersection, yielding a non-compact fiber of $\pi$.

In this addendum, we prove that this property holds if $f:\: (X,\Delta)\to S$ is a good fibration.

\begin{lemma}
    If $f:\: X\to S$ is a good fibration, then $V(X/S)$ is spanned by $f$-exceptional divisors of insufficient fiber type.
    
\end{lemma}

\begin{proof}
    Let $D$ be a prime vertical divisor whose class in $N^1(X/S)$ is nonzero. By \cite[Lemma 3.2]{Kaw97}, the class $[D]$ can be represented as a linear combination of $f$-exceptional divisors. Replace $D$ with any of these $f$-exceptional divisors. The divisor $D$ is degenerate, either of insufficient fiber type or dominating a cycle of codimension at least $2$, by \cref{lem:degenerate}. We show that $D$ is of insufficient fiber type. Each irreducible component of $f^{-1}f(D)$ dominates $f(D)$. Since $0\neq [D]\in N^1(X/S)$, $f^{-1}f(D)$ has at least two components. Therefore, $D$ doesn't contain the fiber over the generic point of $f^{-1}f(D) \to f(D)$.
\end{proof}

\begin{proposition}
\label{prop:movert}
    If $f:\: (X,\Delta)\to S$ is a good fibration, then $V(X/S) \cap \overline{\mathcal{M}}(X/S) = \{0\}$.
    
\end{proposition}

\begin{proof}
    For an arbitrary class nonzero $e\in V(X/S)\cap\overline{\mathcal{M}}(X/S)$, we may represent $e$ with a degenerate divisor $E$ of insufficient fiber type such that $E$ is vertical, effective, and $E\cdot C \geq 0$ for any mov-1 curve $C$ \cite{Leh12b}. Indeed, because any representative $E$ of $e$ maps to a divisor in $S$, we may take $E$ to be effective by pulling back sufficiently positive classes from $S$. Take $\mu\geq 0$ so that either $E' := E-\mu f^*C$ is not supported on the whole fiber or $E-\mu f^*C = 0 \in N^1(X/S)$. If $E-\mu f^*C = 0 \in N^1(X/S)$, then $E=0 \in N^1(X/S)$ as well. If $E'$ is not supported on the whole fiber, by \cite[Lemma 2.9]{Lai11} there is a component of $E'$ covered by a family of mov-$1$ curves intersecting $E'$ negatively, a contradiction. 
\end{proof}

\begin{corollary}
\label{cor:proper}
    For $f:\: (X,\Delta)\to S$ a good fibration, the quotient $$\pi:\: \overline{\mathcal{M}}(X/S)\to N^1(X/S)/V(X/S)$$ is proper.
\end{corollary}

\begin{proof}
 Fix an inner product on $N^1(X/S)$ so that we may realize $\pi$ as the projection to the orthogonal complement $V'\cong N^1(X/S)/V(X/S)$. Since $\overline{\mathcal{M}}(X/S) \cap V(X/S) = \{0\}$, each set $$\mathcal{M}_r := \{x\in\overline{\mathcal{M}}(X/S) \mid  \dist(x,V(X/S))\leq r\},$$ the truncation of  $\overline{\mathcal{M}}(X/S)$ within distance $r$ of $V(X/S)$, is compact. If a sequence of points $x_i$ diverges in $\overline{\mathcal{M}}(X/S)$, then it must leave every $\mathcal{M}_r$. The mapping $\pi$ is the orthogonal projection with kernel $V(X/S)$, so $\dist(x,V(X/S)) = \|\pi(x)\|$ and $\pi(\mathcal{M}_r)$ are the points of magnitude no more than $r$ in $\pi(\overline{\mathcal{M}}(X/S))$. Hence, $\pi(x_i)$ diverges as well. 
\end{proof}

\bibliographystyle{habbvr}
\bibliography{mybib}

\end{document}